\documentclass[revqno,twoside]{amsart}

\usepackage{amsfonts}
\usepackage{amsmath,amssymb,amsthm,amsthm}
\usepackage{mathtools}
\usepackage{etoolbox}
\usepackage{color}
\usepackage[english]{babel}
\usepackage{hyperref}
\usepackage{microtype}
\usepackage{fullpage}
\usepackage{yfonts}

\usepackage{tikz}
\usepackage{pgfplots}
\usepackage{pgfplotstable}
\usepackage{subfig}
\usetikzlibrary{matrix,external,fit,math}

\usepackage{etoolbox}

\DisableLigatures{encoding = *, family = * }

\newtheorem{theorem}{Theorem}[section]
\newtheorem{corollary}[theorem]{Corollary}
\newtheorem{definition}[theorem]{Definition}
\newtheorem{lemma}[theorem]{Lemma}

\newtheorem{remark}[theorem]{Remark}

\numberwithin{equation}{section}

\def\NN{\mathbb{N}}
\def\ZZ{\mathbb{Z}}

\def\RR{\mathbb{R}}
\def\CC{\mathbb{C}}
\def\TT{\mathbb{T}}
\def\O{\mathcal{O}}
\def\zed{\mathcal{Z}}
\def\sg{\sigma}

\newcommand{\norm}[2]{{\left\|#1\right\|}_{#2}}

\newcommand{\Lp}{L^2_p(\TT)}
\newcommand{\Hp}{H^1_p(\TT)}
\newcommand{\Hpd}{H^{-1}_p(\TT)}
\newcommand{\Hs}[1]{H^{#1}_p(\TT)}
\newcommand{\ttilde}[1]{\tilde{\tilde{#1}}}
\newcommand{\ue}[1]{#1^\varepsilon}

\DeclareMathOperator\erfc{erfc}

\title{Null-controllability properties of the wave equation with a second order memory term}

\author{Umberto Biccari}
\address{DeustoTech, University of Deusto, 48007 Bilbao, Basque Country, Spain.}
\address{Facultad de Ingenier\'ia, Universidad de Deusto, Avenida de las Universidades 24, 48007 Bilbao, Basque Country, Spain, +34 944139003 - 3282.}
\email{umberto.biccari@deusto.es, u.biccari@gmail.com}

\thanks{The work of Umberto Biccari was partially supported by the Advanced Grant DYCON (Dynamic Control) of the European Research Council Executive Agency, the Grants MTM2014-52347, MTM2017-92996 and MTM2017-82996-C2-1-R COSNET of MINECO (Spain), by the ELKARTEK project KK-2018/00083 ROAD2DC of the Basque Government, and by the Grant FA9550-18-1-0242 of AFOSR.}

\author{Sorin Micu}
\address{Department of Mathematics, University of Craiova, 200585, Craiova and Institute of Mathematical Statistics and Applied Mathematics, 70700, Bucharest, Romania.}
\email{sd$\_$micu@yahoo.com}



\begin{document} 
\bibliographystyle{acm}
\maketitle

\begin{abstract}
We study the internal controllability of a wave equation with memory in the principal part, defined on the one-dimensional torus $\mathbb{T}=\RR/2\pi\ZZ$. We assume that the control is acting on an open subset $\omega(t)\subset\mathbb{T}$, which is moving with a constant velocity $c\in\mathbb{R}\setminus\{-1,0,1\}$. The main result of the paper shows that the equation is null controllable in a sufficiently large time $T$ and for initial data belonging to suitable Sobolev spaces. Its proof follows from a careful analysis of the spectrum associated with our problem and from the application of the classical moment method.
\end{abstract}

\section{Introduction}\label{into_sect}

Let $\mathbb{T}:=\RR/2\pi\ZZ$ be the one-dimensional torus and let $Q:=(0,T)\times\mathbb{T}$ and $M\in\RR$. This paper is devoted to the analysis of controllability properties of the following wave equation involving a memory term:
\begin{align}\label{wave_mem}
	\begin{cases}
		\displaystyle y_{tt}(t,x)-y_{xx}(t,x) + M\int_0^t y_{xx}(s,x)\,{\rm ds} = \mathbf{1}_{\omega(t)}u(t,x), & (t,x)\in Q
		\\
		y(0,x)=y^0(x),\;\;y_t(0,x)=y^1(x), & x\in\mathbb{T}.
	\end{cases}
\end{align}

In \eqref{wave_mem} the memory enters in the principal part, and the control is applied on an open subset $\omega(t)$ of the domain $\mathbb{T}$ where the waves propagate. The support $\omega(t)$ of the control $u$ at time $t$ moves in space with a constant velocity $c$, that is,
\begin{align*}
	\omega(t) = \omega_0-ct,
\end{align*}
with $\omega_0\subset\mathbb{T}$ a reference set, open and non empty. The control $u\in L^2(\O)$ is then an applied force localized in $\omega(t)$, where
\begin{align*}
	\O:=\Big\{(t,x)\,\big|\, t\in(0,T), x\in\omega(t)\Big\}.
\end{align*}

Moreover, our equation being defined on the torus $\mathbb{T}$, there is no need of specifying the boundary conditions since they will automatically be of periodic type.

Evolution equations involving memory terms appear in several different applications, for modeling natural and social phenomena which, apart from their current state, are influenced also by their history. Some classical examples are viscoelasticity, non-Fickian diffusion and thermal processes with memory (see \cite{pruss2013evolutionary,renardy1987mathematical} and the references therein). 

Controllability problems for evolution equations with memory terms have been studied extensively in the past. Among other contributions, we mention \cite{kim1993control,leugering1984exact,leugering1987exact,loreti2012boundary, loreti2010reachability,mustafa2015control,pandolfi2013boundary,romanov2013exact} which, as in our case, deal with hyperbolic type equations. Nevertheless, in the majority of these works the issue has been addressed focusing only on the steering of the state of the system to zero at time $T$, without considering that the presence of the memory introduces additional effects that makes the classical controllability notion not suitable in this context. Indeed, driving the solution of \eqref{wave_mem} to zero is not sufficient to guarantee that the dynamics of the system reaches an equilibrium. If we were considering an equation without memory, once its solution is driven to rest at time $T$ by a control, then it vanishes for all $t\geq T$ also in the absence of control. This is no longer true when introducing a memory term, which produces accumulation effects affecting the definition of an equilibrium point for the system and its overall stability.

For these reasons, in some recent papers (see, e.g., \cite{chaves2017controllability,lu2017null}), the classical notion of controllability for a wave equation, requiring the state and its velocity to vanish at time $T$, has been extended with the additional condition
\begin{align*}
	\int_0^T y_{xx}(s,x)\,{\rm dx} = 0,
\end{align*}
that is, with the imposition that the control shall \textit{shut down} also the memory effects. 

The present article is thus devoted to the analysis of the controllability of \eqref{wave_mem} in this special context that we just mentioned, which we shall denote \textit{memory-type null controllability}. Besides, as in other related previous works (see, for instance, the recent paper \cite{chaves2014null}), we shall view the wave model \eqref{wave_mem} as the coupling of a wave-like PDE with an ODE. This approach will enhance the necessity of a moving control strategy. Indeed, we will show that the memory-type controllability of the system fails if the support $\mathcal O$ of the control $u$ is time-independent, unless of course one of the following two situations occurs:
\begin{itemize}
	\item the trivial case where the control set coincides with the entire domain, that is $\mathcal O = Q$;
	\item when $M=0$, in which case \eqref{wave_mem} reduces to the one-dimensional wave equation, whose controllability properties are nowadays classical (\cite{lions1988exact}). Thus, in what follows we will always assume $M\neq 0$.
\end{itemize}

This will be justified by showing the presence of waves localized along vertical characteristics which, due to their lack of propagation, contradict the classical {\it Geometric Control Condition} for the observability of wave processes. We mention that this strategy of a moving control has been successfully used in the past in the framework of the structurally damped wave equation and of the Benjamin-Bona-Mahony equation (\cite{martin2013null,rosier2013unique}). 

The model we study is close to the one considered in \cite{lu2017null}, replacing the zero order memory term $	\int_0^t y(s,x)\,{\rm ds}$ by a second order one. Let us mention that, in order to study the controllability properties of our model, we cannot apply directly the arguments from \cite{lu2017null}, where the compactness of the memory term plays a fundamental role in the proof of the observability inequality. The technique we will use is based on spectral analysis and explicit construction of biorthogonal sequences. Although this approach limits our study to a one-dimensional case, it has the additional advantage of offering new insights on the behavior of this type of problems, through the detailed study of the properties of the spectrum. 

This paper is organized as follows. In Section \ref{wp_sect}, we present the functional setting in which we will work and we discuss the well-posedness of our equation. Moreover, we present our main controllability result. In Section \ref{cv_sect}, we give a characterization of the control problem through the adjoint equation associated to \eqref{wave_mem}. Section \ref{spectrum_sect} is devoted to a complete spectral analysis for our problem, which will then be fundamental in the construction of the biorthogonal sequence in Section \ref{bio_sec}, and the resolution of the moment problem. The controllability of our original equation is proved in Section \ref{control_sect}. Finally, in Section \ref{loc_sol_sect} we will present a further discussion on the necessity of a moving control, based on a geometric optics analysis of our equation.

\section{Functional setting, well-posedness and controllability result}\label{wp_sect}

The memory wave equation \eqref{wave_mem} is well posed in a suitable functional setting that we describe below. First of all, let us introduce the space of square-integrable functions on $\mathbb{T}$ with zero mean as
\begin{align}\label{L2p}
    \Lp:=\left\{f:\mathbb{T}\to\CC\,\Big|\int_{-\pi}^\pi |f(x)|^2\,{\rm dx} <+\infty,\,\int_{-\pi}^\pi f(x)\,{\rm dx} =0\right\},
\end{align}
and the corresponding Sobolev space
\begin{align}\label{H1p}
    \Hp:=\left\{f\in\Lp \,\Big|\,\int_{-\pi}^\pi |f_x(x)|^2\,{\rm dx}<+\infty\right\}.
\end{align}

We stress that if $f\in\Lp$, then it can be regarded as a $2\pi$-periodic function in $\mathbb{R}$, that is $f(x+2\pi)=f(x)$ for a.e. $x\in\RR$. Moreover, $\Lp$ and $\Hp$ are Hilbert spaces endowed with the norms
\begin{align*}
    \norm{f}{\Lp}:= \left(\frac{1}{2\pi}\int_{-\pi}^\pi |f(x)|^2\,{\rm dx}\right)^{\frac 12},\;\;\; \norm{f}{\Hp}:=\left(\frac{1}{2\pi}\int_{-\pi}^\pi |f_x(x)|^2\,{\rm dx}\right)^{\frac 12} = \norm{f_x}{\Lp}.
\end{align*}

We denote by $\Hpd = \left(\Hp\right)'$ the dual of $\Hp$ with respect to the pivot space $\Lp$. Then, if $f\in\Hp$ we have  $f_{xx}\in\Hpd$ with
\begin{align*}
    \langle f_{xx},\phi\rangle_{\Hpd,\Hp} :=\int_{-\pi}^\pi f_x\phi_x\,{\rm dx}\;\;\;\forall\,\phi\in\Hp.
\end{align*}

Next, given any integer $2\leq \sg<+\infty$, we define the higher order Sobolev space $\Hs{\sg}$ by recurrence in the following way
\begin{align}\label{Hsp}
    \Hs{\sg}:=\left\{f\in\Hs{\sg-1} \,\Big|\,\int_{-\pi}^\pi |D^\sg f(x)|^2\,{\rm dx} < +\infty\right\},
\end{align}
where with $D^\sg f$ we indicate the derivative of order $\sg$ of the function $f$. We have that also $\Hs{\sg}$ is a Hilbert space, endowed with the norm
\begin{align*}
    \norm{f}{\Hs{\sg}}:=\left(\frac{1}{2\pi}\int_{-\pi}^\pi |D^\sg f(x)|^2\,{\rm dx}\right)^{\frac 12} = \norm{D^\sg f}{\Lp},
\end{align*}

Moreover, we denote by ${\Hs{-\sg} = \left(\Hs{\sg}\right)'}$ the dual of $\Hs{\sg}$ with respect to the pivot space $\Lp$, endowed with the norm
\begin{align*}
    \norm{f}{\Hs{-\sg}}:=\sup_{\underset{\norm{\phi}{\Hs{\sg}}=1}{\phi\in\Hs{\sg}}}\left|\langle f,\phi\rangle_{\Hs{-\sg},\Hs{\sg}}\,\right|.
\end{align*}

We can now prove that our system is well-posed in the functional setting that we have just described. In more detail, we have the following result.

\begin{theorem}\label{wp_thm}
For any $M\neq 0$, $(y^0,y^1)\in\Hp\times \Lp$ and $u\in L^2(\O)$ such that 
\begin{equation}\label{eq:m0}
\int_{-\pi}^\pi \mathbf{1}_{\omega(t)}u(t,x)\,{\rm dx}=0\qquad t\in (0,T),
\end{equation}
system \eqref{wave_mem} admits a unique weak solution $y\in C([0,T];\Hp)\cap C^1([0,T];\Lp)$. Moreover, there exists a positive constant $\mathcal C$, depending only on $T$, such that
\begin{align}\label{norm_est}
    \norm{y}{C([0,T];\Hp)\cap C^1([0,T];\Lp)}\leq \mathcal C\left(\norm{(y^0,y^1)}{\Hp\times\Lp}+\norm{u}{L^2(\O)}\right).
\end{align}
\end{theorem}

\begin{proof}
The proof is almost standard, employing a classical fixed point argument. Denote
\begin{align*}
    \zed:= C([0,T];\Hp)\cap C^1([0,T];\Lp),
\end{align*}
with the following norm
\begin{align*}
    \norm{y}{\zed}:=\left(\norm{e^{-\alpha t}y}{C([0,T];\Hp)}^2 + \norm{e^{-\alpha t}y_t}{C^1([0,T];\Lp)}^2\right)^{\frac 12},
\end{align*}
where $\alpha$ is a positive real number whose value will be given below. Clearly,
\begin{align*}
    e^{-\alpha T}\norm{y}{C([0,T];\Hp)\cap C^1([0,T];\Lp)}\leq\norm{y}{\zed}\leq \norm{y}{C([0,T];\Hp)\cap C^1([0,T];\Lp)}.
\end{align*}

Therefore, $\zed$ is a Banach space with the norm $\norm{\cdot}{\zed}$ and $\zed$ equals $C([0,T];\Hp)\cap C^1([0,T];\Lp)$ algebraically and topologically. Define the map
\begin{align*}
    \mathcal F:\zed&\rightarrow \zed
    \\
    \widetilde{y}&\mapsto \widehat{y},
\end{align*}
where $\hat{y}$ is the solution to \eqref{wave_mem} with $\int_0^t y_{xx}(s)\,{\rm ds}$ being replaced by $\int_0^t \tilde{y}_{xx}(s)\,{\rm ds}$:

\begin{align}\label{wave_mem_fixed}
	\begin{cases}
		\displaystyle \hat{y}_{tt}(t,x)-\hat{y}_{xx}(t,x) =- M\int_0^t \tilde{y}_{xx}(s,x)\,{\rm ds} + \mathbf{1}_{\omega(t)}u(t,x), & (t,x)\in Q
		\\
		\hat{y}(0,x)=y^0(x),\;\;\hat{y}_t(0,x)=y^1(x), & x\in\mathbb{T}.
	\end{cases}
\end{align}

Equation \eqref{wave_mem_fixed} represents the classical wave equations with a non-homogeneous term belonging to $L^2(0,T;\Hs{-1})$ (since $u$ verifies \eqref{eq:m0}). It has a unique solution $\widehat{y}\in C([0,T];\Hp)\cap C^1([0,T];\Lp)$ such that
\begin{align*}
    \norm{\widehat{y}}{\zed} &\leq \norm{\widehat{y}}{C([0,T];\Hp)\cap C^1([0,T];\Lp)}
    \\
    &\leq \mathcal C\left[\norm{(y^0,y^1)}{\Hp\times\Lp} + \norm{u}{L^2(\O)} + \norm{M\int_0^t \widetilde{y}_{xx}(s)\,{\rm ds}}{H_0^1([0,T];\Hs{-1})}\right]
    \\
    &= \mathcal C\left[\norm{(y^0,y^1)}{\Hp\times\Lp} + \norm{u}{L^2(\O)} + |M|\left(\int_0^T\norm{\widetilde{y}_{xx}(t)}{\Hs{-1}}^2\,{\rm dt}\right)^{\frac 12}\right]
    \\
    &= \mathcal C\left[\norm{(y^0,y^1)}{\Hp\times\Lp} + \norm{u}{L^2(\O)} + |M|\norm{\widetilde{y}}{L^2([0,T];\Hp)}\right],
\end{align*}
where $\mathcal C$ is a positive constant depending only on $T$. Hence, $\mathcal F(\zed)\subset\zed$. Next, given $\ttilde{y}\in\zed$, let $\widehat{\widehat{y}}=\mathcal F(\,\ttilde{y}\,)$ be the corresponding solution to \eqref{wave_mem}. Using the above estimate we have
\begin{align*}
    e^{-\alpha t}& \left(\norm{\mathcal F(\,\widetilde{y}\,)(t)-\mathcal F(\,\ttilde{y}\,)(t)}{\Hp} + \norm{\mathcal F(\,\widetilde{y}\,)_t(t)-\mathcal F(\,\ttilde{y}\,)_t(t)}{\Lp}\right)
    \\
    &\leq e^{-\alpha t}\norm{\widehat{y}-\hat{\widehat{y}}}{C([0,t];\Hp)\cap C^1([0,t];\Lp)}
    \\
    &\leq \mathcal C|M|e^{-\alpha t}\left(\int_0^t \norm{\widetilde{y}(s)-\ttilde{y}(s)}{\Hp}^2\,{\rm ds}\right)^{\frac 12}
    \\
    &= \mathcal C|M|\left(\int_0^t e^{-2\alpha (t-s)}\norm{e^{-\alpha s}\left(\widetilde{y}(s)-\ttilde{y}(s)\right)}{\Hp}^2\,{\rm ds}\right)^{\frac 12}
    \\
    &\leq \mathcal C|M|\left(\int_0^t e^{-2\alpha (t-s)}\,{\rm ds}\right)^{\frac 12} \norm{\widetilde{y}-\ttilde{y}}{\zed} = \mathcal C|M|\left(\frac{1-e^{-2\alpha t}}{2\alpha}\right)^{\frac 12} \norm{\widetilde{y}-\ttilde{y}}{\zed}.
\end{align*}
This implies
\begin{align*}
    \norm{\mathcal F(\,\widetilde{y}\,)-\mathcal F(\,\ttilde{y}\,)}{\zed} \leq \frac{\mathcal C|M|}{\sqrt{2\alpha}} \norm{\widetilde{y}-\ttilde{y}}{\zed}.
\end{align*}
Hence, taking $\alpha=2\mathcal C^2M^2$ we obtain
\begin{align*}
    \norm{\mathcal F(\,\widetilde{y}\,)-\mathcal F(\,\ttilde{y}\,)}{\zed} \leq \frac 12 \norm{\widetilde{y}-\ttilde{y}}{\zed},
\end{align*}
i.e. $\mathcal F$ is a contractive map. Therefore, it admits a unique fixed point, which is the solution to \eqref{wave_mem}. Let now $y$ be this unique solution to \eqref{wave_mem}. We have that
\begin{align*}
	\norm{y(t)}{\Hp}^2 + \norm{y_t(t)}{\Lp}^2 \leq\, & \mathcal C\left(\norm{(y^0,y^1)}{\Hp\times\Lp}^2 + \norm{u}{L^2(\O)}^2 + |M|\int_0^t\norm{y(s)}{\Lp}^2\,{\rm ds}\right)
	\\
	\leq\, & \mathcal C\left(\norm{(y^0,y^1)}{\Hp\times\Lp}^2 + \norm{u}{L^2(\O)}^2\right) 
	\\
	&+ \mathcal C|M|\int_0^t\left(\norm{y(s)}{\Hp}^2+\norm{y_s(s)}{\Lp}^2\,{\rm ds}\right).
\end{align*}
Thus, using Gronwall's inequality we get
\begin{align*}
    \norm{y(t)}{\Hp}^2 + \norm{y_t(t)}{\Lp}^2 \leq \mathcal C\left(1 + \mathcal C|M|e^{\mathcal C|M|t}\right)\left(\norm{(y^0,y^1)}{\Hp\times\Lp}^2 + \norm{u}{L^2(\O)}^2\right),
\end{align*}
which gives \eqref{norm_est} and finishes the proof.
\end{proof}

We conclude this section by introducing our main controllability result. We start by recalling that, in the setting of problems with memory terms, the classical notion of controllability, which requires that $y(T,x)=y_t(T,x)=0$, is not completely accurate. Indeed, in order to guarantee that the dynamics can reach the equilibrium, also the memory term has to be taken into account. In particular, it is necessary that also the memory reaches the null value, that is,
\begin{align}\label{eq:in0}
    \int_0^T y_{xx}(s,x)\,{\rm ds} = 0.
\end{align}

If, instead, we do not pay attention to turn off the accumulated memory, i.e. if \eqref{eq:in0} does not hold, 
then the solution $y$ will not stay at the rest after time $T$ as $t$ evolves. Hence, the correct notion of controllability in this framework is given by the following definition (see, e.g., \cite[Definition 1.1]{lu2017null})
\begin{definition}\label{control_def}
Given $0\leq\sigma <\infty$, system \eqref{wave_mem} is said to be memory-type null controllable at time $T$ if for any couple of initial data $(y^0,y^1)\in\Hs{\sg+1}\times\Hs{\sg}$, there exits a control $u\in L^2(\mathcal O)$ verifying \eqref{eq:m0} such that the corresponding solution $y$ satisfies
\begin{align}\label{mem_control}
    y(T,x) = y_t(T,x) = \int_0^T y_{xx}(s,x)\,{\rm ds} = 0,\qquad x\in\mathbb{T}.
\end{align}
\end{definition}
With this definition in mind, the main result of the present work will be the following.
\begin{theorem}\label{control_thm}
Let $M\neq 0$, $c\in\RR\setminus\{-1,0,1\}$, $T>2\pi\left(\frac{1}{|c|}+\frac{1}{|1-c|}+\frac{1}{|1+c|}\right)$, $\omega_0$ a nonempty open set in $\mathbb{T}$ and
\begin{align*}
    \omega(t) = \omega_0-ct,\qquad t\in [0,T].
\end{align*}
For each initial data $(y^0,y^1)\in H^3_p(\mathbb{T})\times H^2_p(\mathbb{T})$ there exists a control $u\in L^2({\mathcal O})$ verifying \eqref{eq:m0} such that the solution $(y,y_t)$ of \eqref{wave_mem} satisfies \eqref{mem_control}.
\end{theorem}

The proof of Theorem \ref{control_thm}, which will be given in Section \ref{control_sect}, will be based on the moment method and on a careful analysis of the spectrum associated to our equation.

\section{Characterization of the control problem: the adjoint system}\label{cv_sect}

This section is devoted to a characterization of the control problem by means of the adjoint equation associated to \eqref{wave_mem}, which is given by the following system
\begin{align}\label{wave_mem_adj}
    \begin{cases}
        \displaystyle p_{tt}(t,x)-p_{xx}(t,x) + M\int_t^T p_{xx}(s,x)\,{\rm ds} + M q_{xx}^0(x) = 0, & (t,x)\in Q
        \\
        p(T,x)=p^0(x),\;\;p_t(T,x)=p^1(x), & x\in\mathbb{T}.
    \end{cases}
\end{align}

The term $Mq^0_{xx}$ takes into account the presence of the memory in \eqref{wave_mem} and the fact that, according to Definition \ref{control_def}, the controllability of our original equation is really reached only if also this memory is driven to zero. We have the following result.

\begin{lemma}\label{control_id_lemma}
Let $0\leq\sigma <\infty$. The equation \eqref{wave_mem} is memory-type null controllable in time $T$ if and only if, for each initial data $(y^0,y^1)\in\Hs{\sg+1}\times\Hs{\sg}$, there exits a control $u\in L^2(\mathcal O)$ verifying \eqref{eq:m0} such that the following identity holds for any $(p^0,p^1,q^0)\in\Lp\times\Hs{-1}\times\Lp$, 
\begin{align}\label{control_id}
    \int_0^T\int_{\omega(t)} u(t,x)\bar{p}(t,x)\,{\rm dx\,dt} = \big\langle y^0(\cdot),p_t(0,\cdot)\big\rangle_{\Hs{1},\Hs{-1}} - \int_{\mathbb{T}} y^1(x)\bar{p}(0,x)\,{\rm dx},
\end{align}
where $p$ is the unique solution to the adjoint system \eqref{wave_mem_adj}.
\end{lemma}

\begin{proof}
We start by remarking that equation \eqref{wave_mem_adj} has a solution in the energy space associated with the initial data, which will be studied in Lemma \ref{sol_adj_lemma}. We multiply \eqref{wave_mem} by $\bar{p}$ and we integrate by parts on $Q$. Taking into account the boundary conditions, we get
\begin{align*}
    \int_0^T\!\!\int_{\omega(t)} u(t,x)\bar{p}(t,x)\,{\rm dx\,dt} =&\, \int_Q\left(y_{tt}(t,x)-y_{xx}(t,x) + M\int_0^t y_{xx}(s,x)\,ds\right)\bar{p}(t,x)\,{\rm dx\,dt}
    \\
    =&\,\int_{\mathbb{T}} \Big(y_t(t,x)\bar{p}(t,x)-y(t,x)\bar{p}_t(t,x)\Big)\Big|_0^T\,{\rm dx} 
    \\
    &+ \int_Q \!y(t,x)\Big(\bar{p}_{tt}(t,x)-\bar{p}_{xx}(t,x)\Big)\,{\rm dx\,dt}
    \\
    &+ M\int_0^T\int_{\mathbb{T}}\left(\int_0^t y_{xx}(s,x)\,ds\right)\bar{p}(t,x)\,{\rm dx\,dt}
\end{align*}
Moreover, we can compute
\begin{align*}
    M\int_0^T\int_{\mathbb{T}} & \left(\int_0^t y_{xx}(s,x)\,ds\right)\bar{p}(t,x)\,{\rm dx\,dt}
    \\
    =&\, -M\int_0^T\int_{\mathbb{T}} \left(\int_0^t y_{xx}(s,x)\,{\rm ds}\right)\frac{{\rm d}}{{\rm dt}}\left(\int_t^T \bar{p}(s,x)\,{\rm ds} + \bar{q}^{\,0}(x)\right)\,{\rm dx\,dt}
    \\
    =&\, -M\int_{\mathbb{T}} \left(\int_0^T y_{xx}(s,x)\,{\rm ds}\right)\bar{q}^{\,0}(x)\,dx + M\int_0^T\int_{\mathbb{T}} y(t,x)\left(\int_t^T \bar{p}_{xx}(s,x)\,{\rm ds} + \bar{q}^{\,0}_{xx}(x)\right)\,{\rm dx\,dt}.
\end{align*}
Summarizing, and taking into account the regularity properties of the solution $y$, we obtain
\begin{align}\label{control_id_prel}
	\int_0^T \int_{\omega(t)} u(t,x)\bar{p}(t,x)\,{\rm dx\,dt} =&\,\int_{\mathbb{T}} y_t(T,x)\bar{p}^{\,0}(x)\,{\rm dx} -\big\langle y(T,\cdot),p^1(\cdot)\big\rangle_{\Hs{1},\Hs{-1}} \notag 
	\\
	& - \int_{\mathbb{T}} y^1(x)\bar{p}(0,x)\,{\rm dx} + \big\langle y^0(\cdot),p_t(0,\cdot)\big\rangle_{\Hs{1},\Hs{-1}} \notag
	\\
	&- M\int_0^T \left\langle y_{xx}(t,\cdot),q^0(\cdot)\right\rangle_{\Hs{-1},\Hs{1}}\,{\rm dt}.
\end{align}
Let us now assume that \eqref{control_id} holds. Then, \eqref{control_id_prel} is rewritten as
\begin{align*}
    \int_{\mathbb{T}} y_t(T,x)\bar{p}^{\,0}(x)\,{\rm dx} -\big\langle y(T,\cdot),p^1(\cdot)\big\rangle_{\Hs{1},\Hs{-1}} -M\int_0^T \left\langle y_{xx}(t,\cdot),q^0(\cdot)\right\rangle_{\Hs{-1},\Hs{1}}\,{\rm dt}=0.
\end{align*}

Since the above identity has to be verified  for all $(p^0,p^1,q^0)\in\Lp\times\Hs{-1}\times\Lp$, we immediately conclude that \eqref{mem_control} holds. 

On the other hand, let us assume that \eqref{wave_mem} is memory-type null controllable, according to Definition \ref{control_def}. Then, from \eqref{control_id_prel} we obtain \eqref{control_id}. The proof is then concluded.
\end{proof}

In addition, in order to prove our controllability result, we will introduce a suitable change of variables, which will allow us to work in a framework in which the control region is fixed (that is, it does not depend on the time variable).

Before doing that, let us observe that both \eqref{wave_mem} and \eqref{wave_mem_adj} may be rewritten as systems coupling a PDE and an ODE. In more detail, we have that the they are equivalent, respectively, to
\begin{align}\label{wave_mem_syst}
    \begin{cases}
        y_{tt}(t,x) - y_{xx}(t,x) +Mz(t,x) = \mathbf{1}_{\omega(t)}u(t,x), &(t,x)\in Q
        \\
        z_t(t,x) = y_{xx}(t,x), &(t,x)\in Q
        \\
        y(0,x) = y^0(x),\;\; y_t(0,x) = y^1(x), & x\in\mathbb{T}
        \\
        z(0,x) = 0, & x\in\mathbb{T},
    \end{cases}
\end{align}
and
\begin{align}\label{wave_mem_adj_syst}
    \begin{cases}
        p_{tt}(t,x) - p_{xx}(t,x) + Mq_{xx}(t,x) = 0, &(t,x)\in Q
        \\
        -q_t(t,x) = p(t,x), &(t,x)\in Q
        \\
        p(T,x) = p^0(x),\;\; y_t(T,x) = p^1(x), & x\in\mathbb{T}
        \\
        q(T,x) = q^0(x), & x\in\mathbb{T}.
    \end{cases}
\end{align}
Let us now consider the change of variables
\begin{align}\label{cv}
	x\mapsto x+ct =: x',
\end{align}
and define
\begin{align*}
    \xi(t,x'):=y(t,x),\;\;\; \zeta(t,x'):=z(t,x),\;\;\; \varphi(t,x'):=p(t,x),\;\;\; \psi(t,x'):=q(t,x).
\end{align*}

The parameter $c$ in \eqref{cv} is a constant velocity which belongs to $\mathbb{R}\setminus\{-1,0,1\}$. It is simply a matter of computations to show that, from \eqref{wave_mem_syst} and \eqref{wave_mem_adj_syst}, after having denoted $x'=x$ with some abuse of notation, we obtain respectively
\begin{align}\label{wave_mem_syst_cv}
    \begin{cases}
        \xi_{tt}(t,x) - (1-c^2)\xi_{xx}(t,x) + 2c\xi_{xt}(t,x) + M\zeta(t,x) = \mathbf{1}_{\omega_0}\tilde{u}(t,x), &(t,x)\in Q
        \\
        \zeta_t(t,x) + c\zeta_x(t,x) = \xi_{xx}(t,x), &(t,x)\in Q
        \\
        \xi(0,x) = y^0(x),\;\; \xi_t(0,x) = y^1(x)-cy^0_x(x), & x\in\mathbb{T}
        \\
        \zeta(0,x) = 0, & x\in\mathbb{T},
\end{cases}
\end{align}
and
\begin{align}\label{wave_mem_adj_syst_cv}
    \begin{cases}
        \varphi_{tt}(t,x) - (1-c^2)\varphi_{xx}(t,x) + 2c\varphi_{xt}(t,x) + M\psi_{xx}(t,x) = 0, &(t,x)\in Q
        \\
        -\psi_t(t,x) - c\psi_x(t,x) = \varphi(t,x), &(t,x)\in Q
        \\
        \varphi(T,x) = p^0(x),\;\; \varphi_t(T,x) = p^1(x)-cp^0_x(x), & x\in\mathbb{T}
        \\
        \psi(T,x) = q^0(x), & x\in\mathbb{T}.
    \end{cases}
\end{align}

Furthermore, in analogy to Lemma \ref{control_id_lemma} we have the following result of characterization of our control problem

\begin{lemma}\label{control_id_lemma_syst}
Let $0\leq\sigma<\infty$. The equation \eqref{wave_mem_syst_cv} is memory-type null controllable in time $T$ if and only if, for each initial data $(y^0,y^1)\in\Hs{\sg+1}\times\Hs{\sg}$, there exits a control $\tilde{u}\in L^2(Q)$ verifying 
\begin{equation}\label{eq:m01}
\int_{-\pi}^\pi \mathbf{1}_{\omega_0}\tilde{u}(t,x)\,{\rm dx}=0\qquad t\in (0,T),
\end{equation}such that the following identity holds for any $(p^0,p^1,q^0)\in\Lp\times\Hs{-1}\times\Lp$, 
\begin{align}\label{control_id_cv}
    \int_0^T\int_{\omega_0} \tilde{u}(t,x)\bar{\varphi}(t,x)\,{\rm dx\,dt }= \big\langle y^0(\cdot),\varphi_t(0,\cdot)\big\rangle_{\Hs{1},\Hs{-1}} - \int_{\mathbb{T}} \big(y^1(x)+cy^0_x(x)\big)\bar{\varphi}(0,x)\,{\rm dx},
\end{align}
where $(\varphi,\psi)$ is the unique solution to \eqref{wave_mem_adj_syst_cv}.
\end{lemma}

\begin{proof}
The proof is totally analogous to the one of Lemma \ref{control_id_lemma}, and we leave it to the reader.
\end{proof}

\section{Spectral analysis}\label{spectrum_sect}

This section is devoted to a spectral analysis of the differential operators corresponding to our adjoint equations, which will be fundamental in the proof of the controllability result. We start by focusing firstly on the original system \eqref{wave_mem_adj_syst} and, in a second moment, we will consider \eqref{wave_mem_adj_syst_cv}.

\subsection{Spectral analysis of \eqref{wave_mem_adj_syst}}

First of all, observe that equation \eqref{wave_mem_adj_syst} can be equivalently written as a system of first order PDEs in the following way
\begin{align}\label{eq:memsys}
    \begin{cases}
        Y'(t) + \mathcal A Y = 0, & t\in(0,T)
        \\
        Y(0)=Y^0,
    \end{cases}
\end{align}
where $Y=\left(\begin{array}{c}
        p
        \\
        r
        \\
        q\end{array}\right)$,
$Y^0=\left(\begin{array}{c}
    p^0
    \\
    r^0
    \\
    q^0\end{array}\right)$ and the unbounded operator ${\mathcal A}: D({\mathcal A})\rightarrow
X_{-\sigma}$ is defined by
\begin{align*}
\mathcal A\left(\begin{array}{c}
    p
    \\
    r
    \\
    q\end{array}\right)=\left(\begin{array}{c}
        -r
        \\
        -p_{xx}+Mq_{xx}
        \\
        p\end{array}\right).
\end{align*}
The spaces $X_{-\sigma}$ and $D({\mathcal A})$ are given respectively by
\begin{align}\label{eq:xsmi}
&X_{-\sigma}=\Hs{-\sg}\times \Hs{-\sg-1}\times \Hs{-\sg},
\\ \label{eq:dsmi}
&D({\mathcal A})=\left\{\left(\begin{array}{c}
    p
    \\
    r
    \\
    q\end{array}\right)\in X{-\sigma}\,\left.\, \right| \,\, p-Mq\in  \Hs{-\sg-1}\right\} .
\end{align}

\begin{theorem}\label{eigen_mu_thm}
The spectrum of the operator $(D({\mathcal A}),\mathcal A)$ is given by
\begin{align}\label{eq:spec}
    \sigma(\mathcal A)=\left(\mu_n^j\right)_{n\in\NN^*,\, 1\leq j\leq 3}\cup \{M\},
\end{align}
where the eigenvalues $\mu_n^j$ verify the following estimates
\begin{align}\label{eq:eig}
    \begin{cases}
        \displaystyle \mu_n^1 = M - \frac{M^3}{n^2} + \mathcal O\left(\frac{1}{n^4}\right), & n\in\NN^*
        \\[15pt]
        \displaystyle \mu_n^2 = -\frac{\mu_n^1}{2} + i\,\sqrt{3\left(\frac{\mu_n^1}{2}\right)^2+n^2} =- \frac M2 +\frac{M^3}{2n^2} + in + i\,\frac{3M^2}{8n} + \mathcal O\left(\frac{1}{n^3}\right), & n\in\NN^*
        \\[15pt]
        \displaystyle \mu_n^3 = \bar{\mu}^2_n, & n\in\NN^*.
    \end{cases}
\end{align}
Each eigenvalue $\mu_n^j\in \sigma({\mathcal A})$ is double and has two associated eigenvectors  \begin{align}\label{eq:eigf}
    \Phi_{\pm n}^j = \left(\begin{array}{c}
        1
        \\
        -\mu_n^j
        \\ [5pt]
        \displaystyle\frac{1}{\mu_n^j}\end{array}\right) e^{\pm inx}, \qquad j\in\{1,2,3\},\,\, n\in\NN^*.
    \end{align}
\end{theorem}

\begin{proof}
From the equality
\begin{align*}
    \mathcal A\left(\begin{array}{c}
        \phi^1
        \\
        \phi^2
        \\
        \phi^3
    \end{array}\right) = \left(\begin{array}{c}
        -\phi^2
        \\
        -\phi^1_{xx} + M\phi^3_{xx}
        \\
        \phi^1
    \end{array}\right) =\mu\left(\begin{array}{c}
        \phi^1
        \\
        \phi^2
        \\
        \phi^3
    \end{array}\right),
\end{align*}
we deduce that
\begin{align*}
    \begin{cases}
        \phi^2 = -\mu\phi^1
        \\
        \phi^3 = \mu^{-1}\phi^1
        \\
        -(\mu-M)\phi^1_{xx} = -\mu^3\phi^1
        \\
        \phi^1\in \Hs{2+\sg}.
    \end{cases}
\end{align*}
Plugging the first two equations from the above system in the third one, we immediately get
\begin{align*}
    -\phi^1_{xx} = \left(-\frac{\mu^3}{\mu-M}\right)\phi^1.
\end{align*}
Consequently, $\phi^1$ takes the form
\begin{align*}
    \phi^1(x) = e^{inx},\;\;\; n\in\ZZ^*,
\end{align*}
and $\mu$ is an eigenvalue of the operator $\mathcal A$ corresponding to $\phi^1$ if and only if verifies the characteristic equation
\begin{align}\label{eq:lambda}
    \mu^3+n^2\mu-Mn^2 = 0.
\end{align}

It is easy to see that, for each $n\geq 1$, \eqref{eq:lambda} has a unique real root located between 0 and $M$. This purely real root will be denoted $\mu^1_n$. 
Since $\mu_n^1$ verifies
\begin{align*}
    \mu_n^1 = M-\frac{M(\mu_n^1)^2}{(\mu_n^1)^2+n^2},
\end{align*}
it follows that
\begin{align}\label{eq:lambdanp}
    \mu_n^1 = M - \frac{M^3}{n^2} + \mathcal O\left(\frac{1}{n^4}\right), \qquad n\geq 1.
\end{align}

Thus, in particular, we have that $M$ is an accumulation point in the spectrum and belongs to the essential part of the spectrum.

We pass to analyze the complex roots of the characteristic equation \eqref{eq:lambda}. Let us consider that $\mu=\alpha+i\beta$ with $\alpha,\beta\in\mathbb{R}$ and $\beta\neq 0$. Plugging this value in \eqref{eq:lambda} and setting both the real an the imaginary part to zero, we have that the following relations hold
\begin{align*}
    \begin{cases}
        \beta^2=3\alpha^2+n^2
        \\
        -8\alpha^3-2\alpha n^2-Mn^2=0.
    \end{cases}
\end{align*}
Hence, $-2\alpha$ is the real solution of the characteristic equation \eqref{eq:lambda}, and we deduce that
\begin{align}
    \begin{cases}
        \displaystyle \mu_n^{2,3} = \alpha_n^{2,3}+i\beta_n^{2,3}, & n\in\NN^*
        \\[7pt]
        \displaystyle \alpha_n^2 = \alpha^3_n = -\frac{\mu_n^1}{2}, & n\in\NN^*
        \\[7pt]
        \displaystyle \beta_n^{2} = \sqrt{3\left(\frac{\mu_n^1}{2}\right)^2+n^2},\quad  \beta_n^3=-\sqrt{3\left(\frac{\mu_n^1}{2}\right)^2+n^2}, & n\in\NN^*.
    \end{cases}
\end{align}

Therefore, we have identified the three families of eigenvalues $\mu_n^j$, $n\in\NN^*$, $j\in\{1,2,3\}$, which behave as in \eqref{eq:eig} (see also Figure \ref{figure_mu}). $M$ is a limit point of the first family $(\mu_n^1)_{n\geq 1}$ and forms the essential part of the spectrum.

\begin{figure}[h]
    \centering
    \includegraphics[scale=0.5]{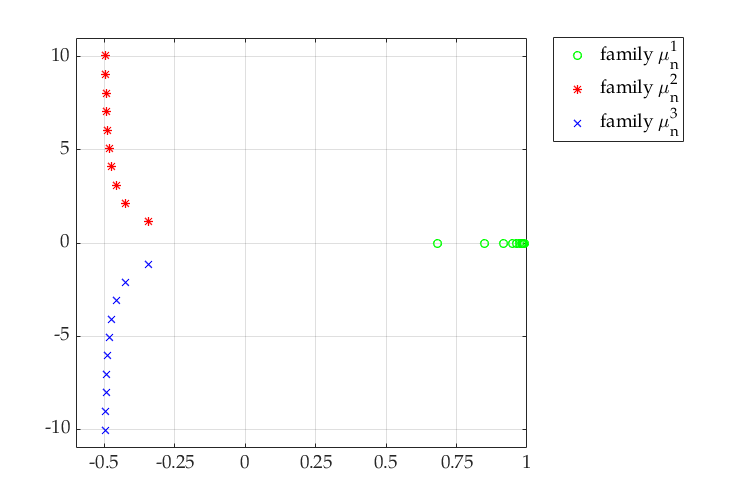}
    \caption{Distribution  of the eigenvalues $\mu_n^j$ for $n\in\{1,\ldots,10\}$, corresponding to $M=1$. The accumulation of the family $\mu^1_n$ zeros at $M$ appears.}\label{figure_mu}
\end{figure}

To each eigenvalue $\mu_n^j$, $j\in\{1,2,3\}$, $n\in\mathbb{N}^*$, correspond two independent eigenfunctions $\Phi^j_{\pm n}$ given
by \eqref{eq:eigf}. \end{proof}

We conclude this section with the following result of the eigenvalues of \eqref{wave_mem_adj_syst}, which will be then needed for the analysis of the spectrum of \eqref{wave_mem_adj_syst_cv}. 

\begin{lemma}\label{lemma:sir} 
The sequence of real numbers $\left(|\mu_n^1|\right)_{n\geq 1}$ is increasing and verifies
\begin{equation}\label{eq:red}
	\frac{|M|}{M^2+1}\leq \left|\mu_n^1\right|<|M|, \qquad n\geq 1.
\end{equation}
Moreover, the sequence $\left(\frac{|\mu_n^1|}{n}\right)_{n\geq 1}$ is decreasing.
\end{lemma}
\begin{proof} 
We recall that $|\mu_n^1|\in(0,|M|)$ is the unique real solution of the equation 
\begin{equation}\label{eq:mumod} 
	\mu^3+n^2\mu-n^2|M|=0, 
\end{equation}
which implies that
\begin{align*}
	\frac{|\mu_{n}^1|}{|M|}=\frac{n^2}{(\mu_{n}^1)^2+n^2}\geq \frac{n^2}{M^2+n^2} \geq \frac{1}{M^2+1},
\end{align*}
and consequently \eqref{eq:red} holds true. Moreover, if $\frac{|\mu_n^1|}{n}\leq \frac{|\mu_{n+1}^1|}{n+1}$, we have that
\begin{align*}
	\frac{|M|}{n}=\frac{|\mu_n^1|^3}{n^3} +\frac{|\mu_n^1|}{n}\leq  \frac{|\mu_{n+1}^1|^3}{(n+1)^3} +\frac{|\mu_{n+1}^1|}{n+1}=\frac{|M|}{n+1},
\end{align*}
which is a contradiction. Hence, the sequence $\left(\frac{|\mu_n^1|}{n}\right)_{n\geq 1}$ is decreasing. Finally, by resting the relations \eqref{eq:mumod} for $|\mu_n^1|$ and $|\mu_{n+1}^1|$ we obtain that
\begin{align*}
	\left(|\mu_{n+1}^1|-|\mu_n^1|\right)\left(|\mu_{n+1}^1|^2+|\mu_{n}^1|^2+|\mu_{n+1}^1||\mu_{n}^1|+(n+1)^2\right)
	=(2n+1)\left(|M|-|\mu_n^1|\right)
\end{align*}
From the last relation and the fact that $|\mu_n^1|\in(0,|M|)$ we deduce that
\begin{align*}
	|\mu_{n+1}^1|>|\mu_n^1|,
\end{align*}
which proves that the sequence  $\left(|\mu_n^1|\right)_{n\geq 1}$ is increasing.
\end{proof}

\begin{remark} The spectral properties from Theorem \ref{eigen_mu_thm} has some immediate interesting consequences:
\begin{enumerate}

\item Since there are eigenvalues with positive real part, \eqref{eq:memsys} is not dissipative. However, since all the real parts of the eigenvalues are uniformly bounded, the well-posedness of the equation is ensured.

\item The existence of a family of eigenvalues having an accumulation point implies that our initial equation \eqref{wave_mem} cannot be controlled with a fixed support control $\omega_0$, unless in the trivial case $\omega_0=\mathbb{T}$.

\end{enumerate}
\end{remark}
\subsection{Spectral analysis of \eqref{wave_mem_adj_syst_cv}}
We now move to the spectral analysis for \eqref{wave_mem_adj_syst_cv}. We recall that $c$ is a parameter which belongs to  $\mathbb{R}\setminus\{-1,0,1\}$. Similarly to the previous section, equation \eqref{wave_mem_adj_syst_cv} can be equivalently written as a system of order one in the following way
\begin{align}\label{eq:memchanged}
    \begin{cases}
        W'(t) + \mathcal A_cW=0, & t\in (0,T)
        \\
        W(T)=W^0,
    \end{cases}
\end{align}
where $W=\left(\begin{array}{c}
    \varphi\\\eta\\\psi\end{array}\right)$,
$W^0=\left(\begin{array}{c} \varphi^0\\\eta^0\\\psi^0\end{array}\right)$ and the unbounded operator ${\mathcal A}_c: D({\mathcal A})\rightarrow X_{-\sigma}$ is defined by
\begin{align*}
    \mathcal A_c\left(\begin{array}{c}
        \varphi
        \\
        \eta
        \\
        \psi\end{array}\right)=\left(\begin{array}{c}
            -\eta
            \\
            -(1-c^2)\varphi_{xx}+2c\eta_x+M\psi_{xx}
            \\
            c\psi_x+\varphi\end{array}\right).
\end{align*}

The domain $D({\mathcal A}_c)$ is the same as $D({\mathcal A})$, $D({\mathcal A}_c)=\Hs{-\sg+1}\times \Hs{-\sg}\times \Hs{-\sg+1}.$ The following two results give the spectral properties of the operator ${\mathcal A}_c$.

\begin{theorem} Let $c\in \mathbb{R}\setminus\{-1,0,1\}$. The spectrum of the operator $(D(\mathcal A_c),\mathcal A_c)$ is given by
\begin{align}\label{eq:specac}
    \sigma(\mathcal A_c)=\left(\lambda_n^j\right)_{n\in\mathbb{Z}^*,\, 1\leq j\leq 3},
\end{align}
where the eigenvalues $\lambda_n^j$ are defined as follows
\begin{align}\label{eq:eigneq}
    \lambda_n^j=icn +\mu_{|n|}^j, \qquad  n\in\mathbb{Z}^*,\,\, 1\leq j\leq 3.
\end{align}
and $\mu_{|n|}^j$ are given in Theorem \ref{eigen_mu_thm}.
\end{theorem}

\begin{proof}
From the equality
\begin{align*}
    \mathcal A_c\left(\begin{array}{c}
    \phi^1
    \\
    \phi^2
    \\
    \phi^3
    \end{array}\right) =\lambda\left(\begin{array}{c}
    \phi^1
    \\
    \phi^2
    \\
    \phi^3
    \end{array}\right),
\end{align*}
we deduce that
\begin{align*}
    \begin{cases}
    \phi^2 = -\lambda\phi^1
    \\
    \lambda\phi^3 = c\phi^3_x+\phi^1
    \\
    -(1-c^2)\phi^1_{xx} - 2c\lambda\phi^1_x + M\phi^3_{xx} = -\lambda^2\phi^1
    \\
    \phi^1\in \Hs{2+\sg}.
    \end{cases}
\end{align*}
To find the solutions of the above system we write
\begin{align*}
    \phi_j(x)=\sum_{n\in\mathbb{Z}^*}a_{nj}e^{inx},\quad 1\leq j\leq 3,
\end{align*}
and we deduce that the coefficients $a_{nj}$ verify
\begin{align*}
    \begin{cases}
        a_{n2}=-\lambda a_{n1}
        \\[15pt]
        \displaystyle a_{n3}=\frac{a_{n1}}{\lambda-cn i}
        \\[15pt]
        \displaystyle \left[(1-c^2)n^2 -2c\lambda n i -\frac{Mn^2}{\lambda - cn i}\right]a_{n1}=-\lambda^2a_{n1}.
    \end{cases}
\end{align*}
Therefore, $\lambda$ is an eigenvalues of ${\mathcal A}_c$ if and only if verifies the equation
\begin{align}\label{eq:lanew}
    (\lambda-icn)^3+n^2(\lambda-icn)-M^2n^2=0.
\end{align}

By taking into account that $\mu_{|n|}^j$ are the zeros of \eqref{eq:lambda}, we obtain that the eigenvalues of the operator ${\mathcal A}_c$ are given by \eqref{eq:eigneq}.
\end{proof}
\begin{figure}[h]
    \centering
    \includegraphics[scale=0.44]{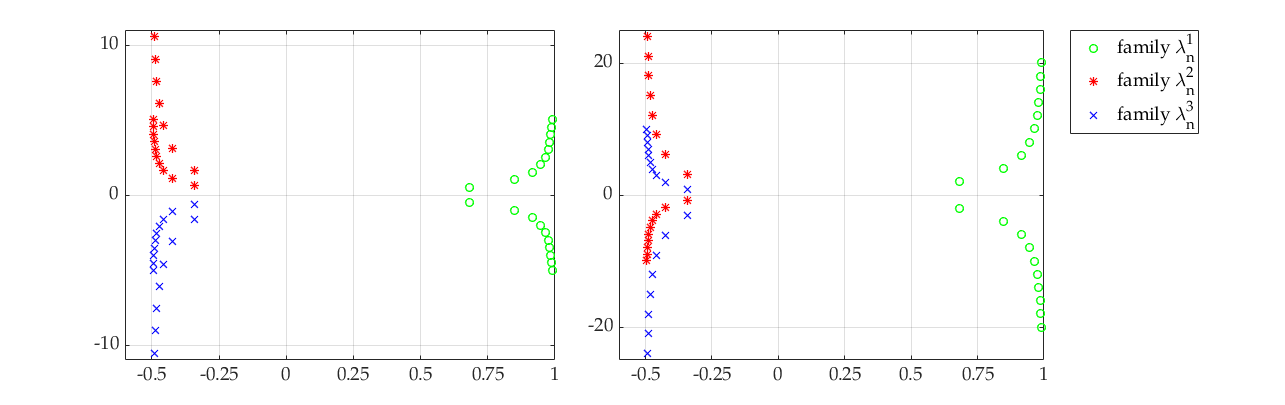}
    \caption{Distribution  of the eigenvalues $\lambda_n^j$ for $n\in\{1,\ldots,10\}$, corresponding to $M=1$ and $c=0.5$ (left) and $c=2$ (right).}\label{figure_lambda}
\end{figure}

\begin{remark} Unlike the operator $(D(\mathcal{A}),\mathcal A)$ from Theorem \ref{eigen_mu_thm}, the operator $(D(\mathcal A_c),\mathcal A_c)$ has a spectrum which contains only eigenvalues of finite multiplicities, if $c\in \mathbb{R}\setminus\{-1,0,1\}$. Notice that, if $c=\pm1$, then there exists an accumulation point in the spectrum of the operator $(D(\mathcal A_c),\mathcal A_c)$. Indeed, from \eqref{eq:eigneq} we deduce that 
\begin{align}\label{eq:lic}
	\begin{cases}
		\displaystyle \lim_{n\rightarrow \infty} \lambda^2_{-n}=\lim_{n\rightarrow \infty} \lambda^3_{n}=-\frac{M}{2}, & \mbox{ if }c=1
		\\[10pt]
		\displaystyle \lim_{n\rightarrow \infty} \lambda^2_{n}=\lim_{n\rightarrow \infty} \lambda^3_{-n}=-\frac{M}{2}, & \mbox{ if }c=-1.
	\end{cases}
\end{align}

As mentioned before, this implies that our initial equation \eqref{wave_mem} cannot be controlled with a control support of the form $\omega(t)=\omega_0 \pm t$, unless in the trivial case $\omega_0=\mathbb{T}$.
\end{remark}

\begin{theorem} \label{te:lari}Each eigenvalue $\lambda_n^j\in \sigma(\mathcal A_c)$
has an associated eigenvector of the form
\begin{align}\label{eq:eigfnew}
    \Psi_{n}^{j}=\left(\begin{array}{c}
                    1
                    \\[5pt]
                    -\lambda_n^j
                    \\[5pt]
                    \displaystyle\frac{1}{\lambda_n^j-icn}\end{array}\right)e^{inx},\qquad
    j\in\{1,2,3\},\,\, n\in\mathbb{Z}^*.
\end{align}
For any $\sigma\geq 0$, the set $\left(|n|^\sigma\Psi^{j}_n\right)_{n\in\mathbb{Z}^*,\,1\leq j\leq 3}$  forms a Riesz basis in the spaces $X_{-\sigma}$.
\end{theorem}
\begin{proof}
It is easy to see that to each eigenvalue $\lambda_n^j$, $j\in\{1,2,3\}$, $n\in\mathbb{Z}^*$, corresponds an eigenfunction $\Psi^j_{n}$ given
by \eqref{eq:eigfnew}. Let us show that $\left(|n|^\sigma\Psi^{j}_n\right)_{n\in\mathbb{Z}^*,\,1\leq j\leq 3}$  form a Riesz basis in the space $X_{-\sigma}$.  Firstly, we remark that $\left(\Psi^{j}_n\right)_{n\in\mathbb{Z}^*,\,1\leq j\leq 3}$ is complete in $X_{-\sigma}$. This follows from the fact that, given an arbitrary element in  $X_{-\sigma}$,
\begin{align*}
    \left(\begin{array}{c}
        y^0
        \\[2pt]
        w^0
        \\[2pt]
        z^0
    \end{array}\right) = \sum_{n\in\mathbb{Z}^*}\left(\begin{array}{c}
            \beta_n^1
            \\[2pt]
            \beta_n^2
            \\[2pt]
            \beta_n^3
        \end{array}\right)|n|^\sigma e^{inx},
\end{align*}
 there exists a unique sequence $(a_n^{j})_{n\in\mathbb{Z}^*,\,1\leq j\leq 3}$ such that
\begin{align*}
    \left(\begin{array}{c}
        y^0
        \\
        w^0
        \\
        z^0
    \end{array}\right) = \sum_{n\in\mathbb{Z}^*,\,1\leq j\leq 3}a_n^{j}|n|^\sigma\Psi_n^{j}.
\end{align*}
From \eqref{eq:eigfnew}, the last relation is equivalent to the system
\begin{align}\label{eq:sysla}
    \begin{cases}
        a_n^{1}+a_n^{2}+a_n^{3}=\beta_n^1, & n\in\mathbb{Z}^*
        \\[5pt]
        \lambda_{n}^1a_n^{1}+\lambda_{n}^2 a_n^{2}+\lambda_{n}^3 a_n^{3}=-\beta_n^2, & n\in\mathbb{Z}^*
        \\[5pt]
        \displaystyle\frac{1}{\lambda_n^1-icn}a_{n}^{1}+\frac{1}{\lambda_n^2-icn}a_{n}^{2}+\frac{1}{\lambda_n^3-icn}a_{n}^{3}=\beta_{n}^3, & n\in\mathbb{Z}^*.
    \end{cases}
\end{align}

The above system has, for each $n\in\mathbb{Z}^*$, a unique solution $\left(a_n^{j}\right)_{1\leq j\leq 3}$, since the determinant of its matrix does not vanish:
\begin{align*}
\det\left|\begin{array}{ccc}
	1& 1& 1
	\\[3pt]
	\lambda_{n}^1 &\lambda_{n}^2 &\lambda_{n}^3
	\\[3pt]
    \displaystyle \frac{1}{\lambda_n^1-icn} & \displaystyle \frac{1}{\lambda_n^2-icn} & \displaystyle \frac{1}{\lambda_n^3-icn}
    \end{array}\right| 
    = \frac{(\lambda_n^2-\lambda_n^1)(\lambda_n^3-\lambda_n^1)(\lambda_n^3-\lambda_n^2)}{(\lambda_n^1-icn)(\lambda_n^2-icn)(\lambda_n^3-icn)}\neq 0.
\end{align*}
To finish the proof of the fact that $\left(|n|^\sigma\Psi^{j}_n\right)_{n\in\mathbb{Z}^*,\,1\leq j\leq 3}$ forms a Riesz basis in $X_{-\sigma}$ we notice that

\begin{align}\label{eq:ri1cv}
    \norm{\sum_{n\in\mathbb{Z}^*,\,1\leq j\leq 3}a_n^j|n|^\sigma\Psi_n^j\,
}{X_{-\sigma}}^2 =& \, \sum_{n\in\mathbb{Z}^*}|n|^{2\sigma}\norm{a_n^1\Psi_n^{1}+a_n^2\Psi_n^{2}+a_n^3\Psi_n^{3}}{X_{-\sigma}}^2 \nonumber 
	\\
	=&\, 2\pi \sum_{n\in\mathbb{Z}^*}\Bigg[\left|a_n^1+a_n^2+a_n^3\right|^2 + 
\frac{1}{n^2}\left|\lambda_{n}^1a_n^1+\lambda_{n}^2a_n^2+
    \lambda_{n}^3a_n^3\right|^2 \nonumber 
    \\
    &+\Bigg|\frac{1}{\lambda_{n}^1-icn}a_n^1+\frac{1}{\lambda_{n}^2-icn}a_n^2+
    \frac{1}{\lambda_{n}^3-icn}a_n^3\Bigg|^2\,\Bigg] \nonumber
    \\
    =&\,2\pi \sum_{n\in\mathbb{Z}^*}\left\|B_n \left(\begin{array}{l}a_n^1\\a_n^2\\a_n^3\end{array}\right)\right\|_2^2,
\end{align}
where the matrix $B_n$ is given by
\begin{align}\label{eq.matrix2}
	B_n=\left(\begin{array}{ccc}
        1& 1& 1
        \\[5pt]
        \displaystyle \frac{\lambda_{n}^1}{|n|} &\displaystyle\frac{\lambda_{n}^2}{|n|}&\displaystyle\frac{\lambda_{n}^3}{|n|}
        \\[10pt]
        \displaystyle \frac{1}{\lambda_n^1-icn} &\displaystyle\frac{1}{\lambda_n^2-icn} &\displaystyle \frac{1}{\lambda_n^3-icn}
    \end{array}\right).
\end{align}
Moreover, we have the following result, whose proof will be given in a second moment.
\begin{lemma}\label{spectrum_B_lemma_cv} 
For each $n\in\mathbb{Z}^*$ let $B_n$ be the matrix defined by \eqref{eq.matrix2}. There exists two positive constants $\textfrak{a}_1$ and $\textfrak{a}_2$, independent of $n$, such that 
\begin{equation}\label{eq:specbn2}
	\sigma(B_n^*B_n)\subset [\textfrak{a}_1,\textfrak{a}_2],\qquad n\in\mathbb{Z}^*.
\end{equation}
\end{lemma}
\noindent Thus, from \eqref{eq:ri1cv} and \eqref{eq:specbn2} we have that the following inequalities hold
\begin{align}\label{eq:insp2}
    \textfrak{a}_1\left(|a_n^1|^2+|a_n^2|^2+|a_n^3|^2\right)\leq
    \norm{B_n\left(\begin{array}{l}a_n^1\\a_n^2\\a_n^3\end{array}\right)}{2}^2\leq
    \textfrak{a}_2\left(|a_n^1|^2+|a_n^2|^2+|a_n^3|^2\right),
\end{align}
which, together with \eqref{eq:ri1cv}, implies that
\begin{align}\label{eq:inri2}
	2\pi \textfrak{a}_1 \sum_{n\in\mathbb{Z}^*,\,1\leq j\leq 3}|a_n^j|^2\leq \norm{\sum_{n\in\mathbb{Z}^*,\,1\leq j\leq 3}a_n^j|n|^\sigma\Psi_n^j\,}{X_{-\sigma}}^2 &\leq 2\pi \textfrak{a}_2 \sum_{n\in\mathbb{Z}^*,\,1\leq j\leq 3}|a_n^j|^2.
\end{align}

Relation \eqref{eq:inri2} shows that $\left(|n|^\sigma\Psi_n^j\right)_{n\in\mathbb{Z}^*,\,1\leq j\leq 3}$ forms a Riesz basis in the space $X_{-\sigma}$ and the proof of Theorem \ref{te:lari} is complete.
\end{proof}

\begin{proof}[Proof of Lemma \ref{spectrum_B_lemma_cv}]
 Let us first remark that $B_n$ is not singular. Indeed as in the case of the matrix of system \eqref{eq:sysla}, we have that
\begin{align}
	\det(B_n)=\frac{(\lambda_n^2-\lambda_n^1)(\lambda_n^3-\lambda_n^1)(\lambda_n^3-\lambda_n^2)}{|n|^3(\lambda_n^1-icn)(\lambda_n^2-icn)(\lambda_n^3-icn)}\neq 0,
\end{align}
which implies that, for each $n\in\mathbb{Z}^*$,
\begin{equation}\label{eq:inspei2}
	\min\left\{|q|\,:\, q\in\sigma(B_n^*B_n)\right\}>0.
\end{equation}
On the other hand, we have that $B_n^*B_n : = \left(B_{k,\ell}\right)_{1\leq k,\ell\leq 3}$ with
\begin{align*}
	&B_{k,k} = 1+\frac{|\lambda_n^k|^2}{|n|^2}+\frac{1}{|\lambda_n^k-icn|^2},\;\;\; k=1,2,3
	\\
	&B_{k,\ell} = 1+\frac{\overline{\lambda}_n^k\lambda_n^\ell}{|n|^2}+\frac{1}{(\overline{\lambda}_n^k+icn)(\lambda_n^\ell-icn)} = \overline{B}_{\ell,k}, \;\;\; k\neq\ell=1,2,3.
\end{align*}	
The above relation, \eqref{eq:eigneq} and \eqref{eq:eig} imply that
\begin{equation}\label{eq:limbn}
	B_n^*B_n \longrightarrow \widetilde{B}:=\left(\begin{array}{ccc}
	\displaystyle 1+c^2+\frac{1}{|M|^2}& 1+c(c+1)& 1+c(c-1)
	\\[4pt]
	\displaystyle 1+c(c+1)& 1+(c+1)^2 &1+(c+1)(c-1)
	\\[4pt]
	\displaystyle 1+c(c-1) &1+(c+1)(c-1)& 1+(c-1)^2
	\end{array}\right)\mbox{ as }n\rightarrow \infty.
\end{equation}

Now, we notice that $\det( \widetilde{B})=\frac{6}{M^2}\neq 0$ and, since the matrix $ \widetilde{B}$ is positive defined, there exist two positive constants $\textfrak{a}_1'$ and $\textfrak{a}_2'$ such that 
\begin{equation}\label{eq:specbn22}
	\sigma( \widetilde{B})\subset [\textfrak{a}_1',\textfrak{a}_2'].
\end{equation}
From \eqref{eq:inspei2} and \eqref{eq:specbn22}, we deduce that \eqref{eq:specbn2} holds and the proof of the Lemma is complete.
\end{proof}


\subsection{Gap properties of the spectrum of \eqref{wave_mem_adj_syst_cv}}

This section is devoted to analyze the distance between the three families composing the spectrum of \eqref{wave_mem_adj_syst_cv}. As it is well-known, these gap properties play a fundamental role in the proof of our controllability result. In what follows, we will use the notation
\begin{align}\label{eq:noset}
	S=\{(n,j)\,:\,n\in\mathbb{Z}^*,\,\, 1\leq j\leq 3\}.
\end{align} 

Since there is no major change in the spectrum if $c$ is replaced by $-c$, we shall limit our analysis to the case $c>0$. Let us start by studying the distance between the elements of $(\lambda_n^1)_{n\in\mathbb{Z}^*}$ and those of  $(\lambda_n^2)_{n\in\mathbb{Z}^*}$ and $(\lambda_n^3)_{n\in\mathbb{Z}^*}$. 

\begin{lemma}\label{lemma:dist1} 
For any $n,m\in\mathbb{Z}^*$ and $k\in\{2,3\}$ we have that 
\begin{equation}\label{eq:diff123} 
\left|\lambda_n^1-\lambda_m^k\right|\geq \frac{|M|}{M^2+1}>0.
\end{equation}
Moreover, for any $n\neq m$, we have that
\begin{equation}\label{eq:diff11}
    \left|\lambda_n^1-\lambda_m^1\right|\geq c.
\end{equation}
\end{lemma}

\begin{proof}
We remark that, if $k\in\{2,3\}$, the numbers  $\Re(\mu_n^1)$ and $\Re(\mu_m^k)$ have oposite signs. By taking into account \eqref{eq:red}, we deduce that
\begin{align*}
	\left|\lambda_n^1-\lambda_m^k\right|\geq\left|\Re\left(\lambda_n^1-\lambda_m^k\right)\right|=\left|\Re\left(\mu_{|n|}^1-\mu_{|m|}^k\right)\right|=\left|\mu_{|n|}^1\right|+\left|\Re(\mu_{|m|}^k)\right|\geq \frac{|M|}{M^2+1},
\end{align*}    
and \eqref{eq:diff123} is proved. On the other hand, since $\mu_k^1\in\mathbb{R}$, for all $n\neq m$ we have that
\begin{align*}
    \left|\lambda_n^1-\lambda_m^1\right|\geq \left|\Im\left(\lambda_n^1-\lambda_m^1\right)\right|=|c(n-m)|\geq c.
\end{align*}
The proof of the Lemma is complete.
\end{proof}

Lemma \ref{lemma:dist1} shows that each element of the sequence $(\lambda_n^1)_{n\in\mathbb{Z}^*}$ is well separated from the others elements of the spectrum of $\mathcal{A}_c$. Let us now analyze the case of the other two families. With this aim, we define the set
\begin{align}\label{spaceV}
	\mathcal V=\left\{ \sqrt{1+3\left(\frac{\mu_n^1}{2n}\right)^2}\,:\,n\in\mathbb{N}^*\right\},
\end{align}
and we pass to analyze the low eigenvalues.

\begin{lemma}\label{lemma:dist20} 
Let $c\in (0,1)\cup (1,\infty)$. 
\begin{enumerate} 
	\item If $c\notin{\mathcal V}$, then for each $N\in\mathbb{N}^*$ there exists a constant $\gamma=\gamma(N,c)>0$ such that the following estimate holds
    \begin{equation}\label{eq:diff023} 
	    \left|\lambda_n^j-\lambda_m^k\right|\geq \gamma(N,c),
	\end{equation}
    for any $(n,j),(m,k)\in S$ with $(n,j)\neq (m,k)$, $1\leq |n|,|m|\leq N$ and $ j,k\in\{2,3\}$.

    \item If $c\in {\mathcal V}$ then there exists a unique $n_c\geq 1$ such that 
    \begin{equation}\label{eq:surprise}
	    \lambda_{-n_c}^2=\lambda_{n_c}^3,
	\end{equation} 
	and for each $N\in\mathbb{N}^*$  there exists a constant $\gamma=\gamma(N,c)>0$ such that the following estimate holds
    \begin{equation}\label{eq:diff023n} 
	    \left|\lambda_n^j-\lambda_m^k\right|\geq \gamma(N,c),
	\end{equation}
    for any $(n,j),(m,k)\in S$ with $(n,j)\neq (m,k)$, $1\leq |n|,|m|\leq N$ and $ j,k\in\{2,3\}$, excepting the case $(n,j)= (- n_c,2)$ and $(m,k)=(n_c,3)$ given by \eqref{eq:surprise}.
\end{enumerate}
\end{lemma}

\begin{proof} 
For any $(n,j),(m,k)\in S$, we have that  
\begin{align*}
    \left|\lambda_{n}^j-\lambda_{m}^k\right|>\left|\Re(\lambda_{n}^j)-\Re(\lambda_{m}^k)\right|=
    \left|\mu_{|n|}^1-\mu_{|m|}^1\right|.
\end{align*}
Since Lemma \ref{lemma:sir} ensures that the sequence of real numbers $\left(|\mu_n^1|\right)_{n\geq 1}$ is increasing, it follows that
\begin{align*}
	\inf \left\{\left|\lambda_{n}^j-\lambda_{m}^k\right|\,:\, (n,j),(m,k)\in S, \,\, |n|\neq |m|,\,\, 1\leq |n|,|m|\leq N,\,\, 2\leq j,k\leq 3\right\}>0.
\end{align*}
It remains to study the case $|n|=|m|$. If $j\in\{2,3\}$ and $n\geq 1$, then 
\begin{align*}
    \left|\lambda_{n}^j-\lambda_{-n}^j\right|>\left|\Im(\lambda_{n}^j)-\Im(\lambda_{-n}^j)\right|=2nc\geq 2c>0.
\end{align*}
Moreover, if $n\geq 1$, then
\begin{align*}
    &\left|\lambda_{n}^2-\lambda_{n}^3\right|>\left|\Im(\lambda_{n}^2)-\Im(\lambda_{n}^3)\right|=2\sqrt{3\left(\frac{\mu_n^1}{2}\right)^2+n^2}\geq 2,
    \\
    &\left|\lambda_{n}^2-\lambda_{-n}^3\right|>\left|\Im(\lambda_{n}^2)-\Im(\lambda_{-n}^3)\right|=2nc+2\sqrt{3\left(\frac{\mu_n^1}{2}\right)^2+n^2}\geq 2,
    \\
    &\left|\lambda_{-n}^2-\lambda_{-n}^3\right|>\left|\Im(\lambda_{-n}^2)-\Im(\lambda_{-n}^3)\right|=2\sqrt{3\left(\frac{\mu_n^1}{2}\right)^2+n^2}\geq 2.
\end{align*}
Finally, we remark that
\begin{align*}
    \left|\lambda_{-n}^2-\lambda_{n}^3\right|>\left|\Im(\lambda_{-n}^2)-\Im(\lambda_{n}^3)\right|=\left|-2nc+2\sqrt{3\left(\frac{\mu_n^1}{2}\right)^2+n^2}\right|.
\end{align*}   
 
The last expression is zero if, and only if, there exists $n_c>0$ such that $c=\sqrt{1+3\left(\frac{\mu_n^1}{2n_c}\right)^2}$ (hence, $c\in{\mathcal V}$), and the proof is complete.
\end{proof}

\begin{remark}  
The previous Lemma shows that, if $c\in{\mathcal V}$ there exists a unique double  eigenvalue $\lambda_{-n_c}^2$. Its geometric multiplicity is two, since there are two  corresponding linearly independent eigenfunctions $\Psi^2_{- n_c}$ and $\Psi^3_{n_c}$.
\end{remark}

Now we pass to analyze the high part of the spectrum. The properties of the large elements of the sequences $(\lambda_n^2)_{n\in\mathbb{Z}^*}$ and $(\lambda_n^3)_{n\in\mathbb{Z}^*}$ are described in the following lemma.

\begin{lemma} \label{lemma:inc} 
For any $\epsilon>0$ there exists $N_\epsilon^1\in\mathbb{N}$ such that
\begin{enumerate}
    \item If $c\in(0,1)$ then we have that
    \begin{itemize}
        \item The sequences $\left(\Im(\lambda_n^2)\right)_{n\geq 1}$ and $\left(\Im(\lambda_{-n}^2)\right)_{n\geq N_\epsilon^1}$ are increasing and included in the intervals $\left[1+c ,\infty\right)$ and $\left[1-c ,\infty\right)$, respectively.

        \item The sequences $\left(\Im(\lambda_n^3)\right)_{n\geq N_\epsilon^1}$ and $\left(\Im(\lambda_{-n}^3\right)_{n\geq 1}$ are decreasing and included in the intervals $\left(-\infty,-1+c\right]$ and $\left(-\infty,-c-1\right]$, respectively.
    \end{itemize}
    Moreover, we have that
    \begin{equation}\label{eq:dinm1}
	    \begin{array}{ll}
		    \Im(\lambda_{n+1}^2)-\Im(\lambda_{n}^2)=\Im(\lambda_{-n}^3)-\Im(\lambda_{-n-1}^3)\geq c+1-\epsilon, & \mbox{ for }n\geq N_\epsilon^1, 
	        \\[8pt]
	        \Im(\lambda_{-n-1}^2)-\Im(\lambda_{-n}^2)=\Im(\lambda_{n}^3)-\Im(\lambda_{n+1}^3)\geq 1-c-\epsilon, & \mbox{ for }n\geq N_\epsilon^1,  
	        \\[8pt]
	        \Im(\lambda_{-n}^2)\leq \Im(\lambda_{-N_\epsilon^1}^2),\quad \Im(\lambda_{n}^3)\geq \Im(\lambda_{N_\epsilon^1}^3),  & \mbox{ for }1\leq n\leq N_\epsilon^1.
        \end{array}
    \end{equation}

    \item If $c\in(1,\infty)$ then we have that
    \begin{itemize}
        \item The sequences $\left(\Im(\lambda_n^2)\right)_{n\geq 1}$ and $\left(\Im(\lambda_{n}^3)\right)_{n\geq N_\epsilon^1}$ are increasing and included in the intervals $\left[c +1,\infty\right)$ and $\left[c-1-\epsilon ,\infty\right)$, respectively.

        \item The sequences $\left(\Im(\lambda_{-n}^2)\right)_{n\geq N_\epsilon^1}$ and $\left(\Im(\lambda_{-n}^3\right)_{n\geq 1}$ are decreasing and included in the intervals $\left(-\infty,1-c+\epsilon\right)$ and $\left(-\infty,-1-c\right)$, respectively.
    \end{itemize}
    Moreover, we have that
    \begin{equation}\label{eq:dinm2}
        \begin{array}{ll}
	        \Im(\lambda_{n+1}^2)-\Im(\lambda_{n}^2)=\Im(\lambda_{-n}^3)-\Im(\lambda_{-n-1}^3)\geq c +1-\epsilon, & \mbox{ for }n\geq N_\epsilon^1, 
	        \\[8pt]
	        \Im(\lambda_{-n}^2)-\Im(\lambda_{-n-1}^2)=\Im(\lambda_{n}^3)-\Im(\lambda_{n+1}^3)\geq c-1-\epsilon, & \mbox{ for }n\geq N_\epsilon^1,
	        \\[8pt]
	        \Im(\lambda_{n}^3)\leq \Im(\lambda_{N_\epsilon^1}^3),\quad \Im(\lambda_{-n}^2)\geq \Im(\lambda_{-N_\epsilon^1}^2),  & \mbox{ for }1\leq n\leq N_\epsilon^1.
        \end{array}
    \end{equation}
    \end{enumerate}
\end{lemma}

\begin{proof}
We first notice that
\begin{equation}\label{eq:leg23}
    \overline{\lambda}_n^2=-icn+\mu^1_{|n|}-i\sqrt{3\left(\frac{\mu^1_{|n|}}{2}\right)^2+n^2}=\lambda_{-n}^3,\qquad n\in\mathbb{Z}^*.
\end{equation}

We deduce that $\Im(\lambda^3_{n})=-\Im(\lambda^2_{-n})$.  Also, let us remark that, since the sequence $\left(\mu_{n}^1\right)_{n\geq 1}$ is bounded (see \eqref{eq:red} above), given any $\epsilon>0$ there exists $N_\epsilon^1\in\mathbb{N}$ such that 
\begin{align}\label{eq:neps}
    \frac{\frac{3}{4}\left(\mu^1_{n}\right)^2}{\sqrt{\frac{3}{4}\left(\mu^1_{n}\right)^2+n^2}+n}\leq \epsilon,\qquad n\geq N_\epsilon^1.
\end{align}
We analyze the cases $c\in(0,1)$ only, the other one being similar. We have that
\begin{itemize}
    \item For the sequence $\left(\Im(\lambda_n^2)\right)_{n\geq 1}$ we have that
    \begin{align*}
        \Im(\lambda_{n}^2)= cn+\sqrt{3\left(\frac{\mu^1_{|n|}}{2}\right)^2+n^2}\geq 1+c, \qquad n\geq 1,
    \end{align*}    
    and, according to \eqref{eq:neps}, for any $n\geq N_\epsilon^1$ it follows that
    \begin{align*}
        \Im(\lambda_{n+1}^2)-\Im(\lambda_{n}^2) &=1 +c+\frac{\frac{3}{4}\left(\mu^1_{|n+1|}\right)^2}{\sqrt{\frac{3}{4}\left(\mu^1_{|n+1|}\right)^2+(n+1)^2}+(n+1)}-\frac{\frac{3}{4}\left(\mu^1_{|n|}\right)^2}{\sqrt{\frac{3}{4}\left(\mu^1_{|n|}\right)^2+n^2}+n}
	    \\
	    &\geq 1+c-\epsilon.
	\end{align*}
            
	\item  As for the sequence $\left(\Im(\lambda_{-n}^2)\right)_{n\geq 1}$ we remark that
	\begin{align*}
	    \Im(\lambda_{-n}^2)= (1-c)n+\frac{\frac{3}{4}\left(\mu^1_{|n|}\right)^2}{\sqrt{\frac{3}{4}\left(\mu^1_{|n|}\right)^2+n^2}+n}\geq 1-c,
	\end{align*}     
	and, according to \eqref{eq:neps}, for any $n\geq N_\epsilon^1$ it follows that
	\begin{align*}
	    \Im(\lambda_{-n-1}^2)-\Im(\lambda_{-n}^2) &= 1 -c+\frac{\frac{3}{4}\left(\mu^1_{|n+1|}\right)^2}{\sqrt{\frac{3}{4}\left(\mu^1_{|n+1|}\right)^2+(n+1)^2}+(n+1)}-\frac{\frac{3}{4}\left(\mu^1_{|n|}\right)^2}{\sqrt{\frac{3}{4}\left(\mu^1_{|n|}\right)^2+n^2}+n}
	    \\
	    &\geq 1-c-\varepsilon.
	\end{align*}
\end{itemize}
	
Hence, all the desired properties of the sequence  $\left(\Im(\lambda_{n}^2)\right)_{n\in\mathbb{Z}^*}$ are proved in the case $c\in (0,1)$. Since $\Im(\lambda_n^3)=-\Im(\lambda_{-n}^2)$, the properties of the sequence  $\left(\Im(\lambda_{n}^3)\right)_{n\in\mathbb{Z}^*}$  follow immediately from those of  $\left(\Im(\lambda_{n}^2)\right)_{n\in\mathbb{Z}^*}$.
\end{proof}

Now we can pass to study the possible interactions between the large elements of the sequences $\left(\lambda_{n}^2\right)_{n\in\mathbb{Z}^*}$  and $\left(\lambda_{n}^3\right)_{n\in\mathbb{Z}^*}$.

\begin{lemma}\label{lemma.dist23} 
Let $\epsilon>0$ sufficiently small. There exists $N_\epsilon\in\mathbb{N}$ and two positive constants $\gamma(c,\epsilon)$ and $\gamma'(c,\epsilon)$ with the property that for each $m\geq N_\epsilon$ there exists $n_m\geq N_\epsilon$ such that
\begin{enumerate}
    \item If $c\in (0,1)$ then
    \begin{equation}\label{eq:di0}
        \frac{1-c}{2}+3\epsilon \geq  \left|\lambda_{m}^2-\lambda_{-n_m}^2\right|=\left|\lambda_{-m}^3-\lambda_{n_m}^3\right|\geq \frac{\gamma'(c,\epsilon)}{|m|^2},
    \end{equation}
    and
    \begin{align}\label{eq:di1}
        \inf\left\{\left|\lambda_{m}^2-\lambda_{n}^2\right|\,:\right.&\left.\,|n|\geq N_\epsilon,\,\, n\neq -n_m \right\} \notag 
        \\
        &=\inf\left\{\left|\lambda_{-m}^3-\lambda_{n}^3\right|\,:\,|n|\geq N_\epsilon,\,\, n\neq n_m \right\}\geq \gamma(c,\epsilon).
    \end{align}

    \item If $c\in (1,\infty)$ then
    \begin{equation}\label{eq:di3}
        \frac{c-1}{2}+3\epsilon\geq  \left|\lambda_{m}^2-\lambda_{n_m}^3\right|=\left|\lambda_{-m}^3-\lambda_{-n_m}^2\right|\geq \frac{\gamma'(c,\epsilon)}{|m|^2},
    \end{equation}
    and
    \begin{align}\label{eq:di4}
        \inf\left\{\left|\lambda_{m}^2-\lambda_{n}^3\right|\,:\right.&\left.\,|n|\geq N_\epsilon,\,\, n\neq n_m \right\}\notag 
        \\
        &=\inf\left\{\left|\lambda_{-m}^3-\lambda_{n}^2\right|\,:\,|n|\geq N_\epsilon,\,\, n\neq- n_m \right\}\geq \gamma(c,\epsilon).
    \end{align}
\end{enumerate}
\end{lemma}
\begin{proof} 
Firstly, notice that according to \eqref{eq:eig} there exists $N^2_\epsilon$ such that 
\begin{equation}\label{eq:remic}
	\left|\Re(\lambda_n^j)+\frac{M}{2}\right|\leq \frac{\epsilon}{2},\qquad |n|\geq N^2_\epsilon,\,\, j\in\{2,3\}.
\end{equation}

Let $m \geq N_\epsilon:=\max\{N^1_\epsilon,N^2_\epsilon\}$, where $N^1_\epsilon$ is the number given by Lemma \ref{lemma:inc}. Let $n_m\in\mathbb{N}$ be such that
\begin{equation}\label{eq:nm}
    \big|(1+c)m-|1-c|n_m\big|=\inf_{n\geq 1}\big{|}(1+c)m-|1-c|n\big{|}.
\end{equation}
Notice that $n_m$ also verifies
\begin{equation}\label{eq:nmint}
    -\frac{1}{2}+\frac{1+c}{|1-c|}m\leq n_m\leq \frac{1}{2}+\frac{1+c}{|1-c|}m,
\end{equation}
and represents nearest integer to $\frac{1+c}{|1-c|}m$. We analyze separately the following two cases:
\begin{enumerate} 
	\item Let $c\in(0,1)$. For each $m\geq N_\epsilon$ we have that
    \begin{equation}\label{eq:innm}
	    \left|\lambda_{m}^2-\lambda_{-n_m}^2\right|=\inf_{n\in \mathbb{Z}^*}\left|\lambda_{m}^2-\lambda_{n}^2\right|,
    \end{equation}
    with $n_m$ given by \eqref{eq:nm}. Indeed, we have that
    \begin{align*}
	    \left|\lambda_{m}^2-\lambda_{-n_m}^2\right| \leq &\, \left|\Im(\lambda_{m}^2-\lambda_{-n_m}^2)\right|+\left|\Re(\lambda_{m}^2-\lambda_{-n_m}^2)\right|
	    \\
	    \leq &\, \left|(1-c)n_m-(1+c)m\right|+\left|(1+c)m-\Im(\lambda_{m}^2)\right|
	    \\
	    &+\left|\Im(\lambda_{-n_m}^2)+(1-c)n_m\right|+\left|\Re(\lambda_{m}^2-\lambda_{-n_m}^2)\right|,
    \end{align*}
    from which, by taking into account \eqref{eq:neps} and \eqref{eq:remic}, we deduce that
    \begin{equation}\label{eq:in1nm}
        \left|\lambda_{m}^2-\lambda_{-n_m}^2\right|\leq  \left|(1+c)m-(1-c)n_m\right|+3\epsilon.
    \end{equation}
    On the other hand, from \eqref{eq:dinm1} and \eqref{eq:neps} we deduce that
    \begin{align*}
        |\lambda_{m}^2&-\lambda_{n}^2| 
        \\
        &\geq \min\left\{\left|\lambda_{m}^2-\lambda_{|n|}^2\right|,\,\,\left|\lambda_{m}^2-\lambda_{-|n|}^2\right| \right\} \geq \min\left\{c+1-\epsilon,\,\, \left|\Im(\lambda_{m}^2-\lambda_{n}^2)\right| \right\}
        \\
        &\geq \min\left\{c+1-\epsilon,\,\, \left|(1+c)m-(1-c)|n|\right|-\left|(1+c)m-\lambda_{m}^2)\right|-\left|\Im(\lambda_{-|n|}^2)+(1-c)|n|\right|\right\}
        \\
        &\geq \min\left\{c+1-\epsilon,\,\, \left|(1+c)m-(1-c)|n|\right|-2\epsilon\right\},
    \end{align*}
    which, by taking into account \eqref{eq:nm}, implies that
    \begin{align}\label{eq:in2nm}
        |\lambda_{n}^2-\lambda_{m}^2| \geq \min\left\{ \left|(1+c)m-(1-c)n_m\right|+2-\epsilon,\,\, \frac{1-c}{2}+ \left|(1+c)m-(1-c)n_m\right|-2\epsilon\right\}.\notag
        \\
    \end{align}
    
    From \eqref{eq:in1nm}-\eqref{eq:in2nm} we deduce that, for $\epsilon$ sufficiently small, \eqref{eq:innm} holds true. It follows that, for each $m\geq N_\epsilon$, we have that
    \begin{equation}
        \inf_{n\in\mathbb{Z}^*,\,\, n\neq n_m}\left|\lambda_{m}^2-\lambda_{n}^2\right|\geq \gamma(\epsilon,c):=\min\left\{2-\epsilon,\,\, \frac{1-c}{2}-2\epsilon\right\},
    \end{equation}
    and
    \begin{equation}
        \left|\lambda_{m}^2-\lambda_{-n_m}^2\right|\geq \left|\Re(\lambda_{m}^2)-\Re(\lambda_{-n_m}^2)\right|\geq \frac{\gamma'(c,\epsilon)}{m^2}.
    \end{equation}

    \item If $c\in(1,\infty)$. As before, for $\epsilon$ small enough and for each $m\geq N_\epsilon$, we have that
    \begin{equation}
        \left|\lambda_{m}^2-\lambda_{n_m}^3\right|=\inf_{n\in\mathbb{Z}^*}\left|\lambda_{m}^2-\lambda_{n}^3\right|,
    \end{equation}
    with $n_m$ given by \eqref{eq:nm} and consequently
    \begin{equation}
        \inf_{n\in\mathbb{Z}^*}\left|\lambda_{m}^2-\lambda_{n}^3\right|\geq \left|\Re(\lambda_{m}^2)-\Re(\lambda_{n_m}^3)\right|\geq \frac{\gamma'(c,\epsilon)}{m^2}.
    \end{equation}
    The rest of the proof is similar to the case $c\in(0,1)$.
\end{enumerate}
\end{proof}

\begin{remark} Let us briefly explain how does the asymptotic spectral gap depend on the algebraic properties of $c$ (to simplify, we suppose that $c\in(1,\infty)$). Let $\mathcal{D}$ denote the set of the real numbers $\theta$ such that 
\begin{equation}\label{eq:ind1}
 0<\left|\theta-{\frac {p}{q}}\right|<{\frac {1}{q^{3}}},
\end{equation}
is satisfied by an infinite number of integer pairs $(p, q)$ with $q > 0$. It is known that  (see, for instance, \cite[Theorem 3.4]{Bugeaud}) $\mathcal{D}$ is an uncountable set (of zero Lebesgue measure). 
\
If $\frac{1+c}{|1-c|}\in \mathcal{D}$, then there exists an infinite number of values $m\in\mathbb{N}$ such that 
\begin{equation}\label{eq:ind2}
 	\Big|(1+c)m-|1-c|n_m \Big|<{\frac {|1-c|}{m^{2}}},
\end{equation}
where $n_m$ verifies \eqref{eq:nm}.  It follows that, for the values $c\in(1,\infty)$ of the velocity having the property $\frac{1+c}{|1-c|}\in \mathcal{D}$, there exists a constant $\gamma''(c,\epsilon)>0$ such that the inequality 
\begin{equation}
\left|\lambda_m ^2-\lambda_{n_m}^3\right|\leq \frac{\gamma''(c,\epsilon)}{m^2},
\end{equation}
is verified by an infinite number of indices $m\in\mathbb{N}$. Consequently, the asymptotic gap of the eigenvalues of the operator $(D(\mathcal{A}_c),\mathcal{A}_c)$ is equal to zero. From the controllability point of view, this implies that there are initial data in the energy space $\Hp\times \Lp$ which cannot be led to zero in any time $T>0$.
\end{remark}

\begin{remark} To summarize the results of this section let us mention that Lemmas \ref{lemma:dist1}-\ref{lemma.dist23} prove that all the elements of the spectrum $\sigma(\mathcal A_c)=
\left(\lambda_n^j\right)_{(n,j)\in S}$ are well separated one from another except for the following special cases:

\begin{enumerate}
	\item If $c\in (0,1)$  the eigenvalues $\lambda^2_m$ and $\lambda^2_{-n_m}$ have a distance at least of order $\frac{1}{m^2}$ between them and a similar relation holds for $\lambda^3_{m}$ and $\lambda^3_{-n_m}$.

	\item If $c\in(1,\infty)$ the eigenvalues $\lambda^2_m$ and $\lambda^3_{n_m}$ have a distance at least of order $\frac{1}{m^2}$ between them and a similar relation holds for $\lambda^3_{-m}$ and $\lambda^2_{-n_m}$.

	\item If $c\in{\mathcal V}$, there exists a unique double eigenvalues  $\lambda_{-n_c}^2=\lambda_{n_c}^3$.
\end{enumerate}

Even though the asymptotic gap between the elements of the spectrum is equal to zero, the fact that we know the velocities with which the distances between these eigenvalues tend to zero, will allow us to estimate the norm of the biorthogonal to the family of exponential functions $\left(e^{-\lambda_n^j t}\right)_{(n,j)\in S}$.
\end{remark}

\section{The biorthogonal family}\label{bio_sec}

In this section we construct and evaluate a biorthogonal sequence $\left(\theta_m^k\right)_{(m,k)\in S}$ in $L^2\left(-\frac{T}{2},\frac{T}{2}\right)$ to the family of exponential functions $\Lambda=\left(e^{-\lambda_{n}^jt}\right)_{(n,j)\in S},$ where $\lambda_n^j$ are given by \eqref{eq:eigneq}. In order to avoid
the double eigenvalue, which according to Lemma \ref{lemma:dist20} occurs if $c\in\mathcal{V}$, and to keep the notation as simple as possible, we make the convention that, if $c\in\mathcal{V}$, we redefine $ \lambda_{-n_c}^2$ as follows
\begin{align}\label{eq:condob}
	\lambda_{-n_c}^2=-cn_ci+i\sqrt{3\left(\frac{\mu_{n_c}^1}{2}\right)^2+n_c^2}-\frac{1}{2}i+\frac{\mu_{n_c}^1}{2}.
\end{align}

In this way Lemmas \ref{lemma:dist1},  \ref{lemma:dist20} and \ref{lemma.dist23} guarantee that all the elements of the family $\left(\lambda_n^j\right)_{(n,j)\in S}$ are different. Since the biorthogonal sequence has the property that
\begin{align*}
	\int_{-\frac{T}{2}}^{\frac{T}{2}}\theta_m^k(t)e^{-\overline{\lambda}_{n}^jt}\,{\rm dt}=\delta_{mk}^{nj},	
\end{align*}
if we define the Fourier transform of $\theta_m^k$,
\begin{align*}
	\widehat{\,\theta}_m^k(z)=\int_{-\frac{T}{2}}^{\frac{T}{2}}\theta_m^k(t)e^{-izt}\,{\rm dt},
\end{align*}
we obtain that
\begin{align}\label{eq:cotr}
	\widehat{\,\theta}_m^k(-i\overline{\lambda}_n^j)=\delta_{mk}^{nj},\qquad (n,j),(m,k)\in S.
\end{align} 
Therefore, we define the infinite product
\begin{align}\label{eq:prod}
	P(z)=z^3\prod_{(n,j)\in S}\left(1+\frac{z}{i\overline{\lambda}_n^j}\right),
\end{align}
and we study some of its properties in the following theorem. The infinite product in definition \eqref{eq:prod} should be understood in the following sense
\begin{align*}
	P(z)=z^3\lim_{R\rightarrow\infty}\prod_{\substack{(n,j)\in S\\|\lambda_n^j|\leq R}}\left(1+\frac{z}{i\overline{\lambda}_n^j}\right).
\end{align*}

We shall prove that the limit exists and defines an entire function. This and other important properties of $P(z)$ are given in the following theorem.

\begin{theorem}\label{te:pprod}
Let $c\in\mathbb{R}\setminus\{-1,0,1\}$ and let $P$ be given by \eqref{eq:prod}. We have that:
\begin{enumerate}
	\item $P$ is well defined, and it is an entire function of exponential type $\left(\frac{1}{|c|}+\frac{1}{|1+c|}+\frac{1}{|1-c|}\right)\pi$.

	\item For each $\delta>0$, there exists a positive constant $C(\delta)>0$ such that $P$ verifies the following estimate 
	\begin{equation}\label{eq:preal}
		\left|P(z)\right|\leq C(\delta),\qquad z=x+iy,\,\,x,y\in\mathbb{R},\,\,|y|\leq \delta.
	\end{equation}  
	Moreover, there exists a constant $C_1>0$ such that, for any $(m,k)\in S$, the following holds 
	\begin{equation}\label{eq:preal2}
		\left|\frac{P(x)}{x+i\overline{\lambda}_m^k}\right| \leq \frac{C_1}{1+\left|x+\Im(\lambda_m^k)\right|},\qquad z=x+iy,\,\,x,y\in\mathbb{R},\,\,|y|\leq \delta.
	\end{equation}

	\item Each point $-i\overline{\lambda}_m^j$ is a simple zero of $P$ and there exists a positive constant $C_2$ such that
	\begin{equation}\label{eq:der}
		\left|P'(-i\overline{\lambda}_m^j)\right|\geq \frac{C_2}{m^2},\qquad (m,j)\in S.
	\end{equation}
\end{enumerate}

\begin{proof} 
Since if we replace $c$ by $-c$ in the eigenvalues expression \eqref{eq:eigneq} we obtain the same spectrum, it is sufficient to study the case $c>0$ only. Let us define the following sequences:
\begin{equation}\label{eq:wnj}
	\begin{cases}
	\displaystyle \nu_n^1=\begin{array}{ll}\frac{1}{c}\lambda_n^1 & n\in\mathbb{Z}^*\end{array}
	\\
	\\
	\nu_n^2=\begin{cases} 
				\frac{1}{c+1}\lambda_n^2 & n\geq 1
				\\
				\frac{1}{c+1}\lambda_{n}^3 & n\leq -1
			\end{cases}
	\\
	\\
	\nu_n^3=\begin{cases}
				\frac{1}{c-1}\lambda_{n}^3 & n\geq 1
				\\
				\frac{1}{c-1}\lambda_{n}^2 & n\leq -1.
			\end{cases}
	\end{cases}
\end{equation}
We notice that, according to \eqref{eq:wnj} and \eqref{eq:eigneq}, $\nu_{-n}^j=\overline{\nu}_{n}^j$, $(n,j)\in S.$ Moreover, we have that
\begin{equation}\label{eq:wnjnew}
	\begin{cases}
		\displaystyle \nu_n^1=i n+\frac{M}{c}+{\mathcal O}\left(\frac{1}{n}\right), & n\in\mathbb{Z}^*
		\\ 
		\\
		\displaystyle \nu_n^2=i n - \frac{M}{2(c+1)}+{\mathcal O}\left(\frac{1}{n}\right), & n\in\mathbb{Z}^*
		\\
		\\
		\displaystyle \nu_n^3=in-\frac{M}{2(c-1)}+{\mathcal O}\left(\frac{1}{n}\right), & n\in\mathbb{Z}^*.
	\end{cases}
\end{equation}
The product $P$ may be rearranged in the following way 
\begin{align}\label{eq:newp}
	P(z)=z^3\prod_{n\in\mathbb{Z}^*}\left(1+\frac{z}{ic\overline{\nu}_n^1}\right) \prod_{n\in\mathbb{Z}^*}\left(1+\right.&\left.\frac{z}{i(c+1)\overline{\nu}_n^2}\right)
\prod_{n\in\mathbb{Z}^*}\left(1+\frac{z}{i(c-1)\overline{\nu}_n^3}\right) \notag 
	\\
	&:=c(c+1)(c-1)P_1\left(\frac{z}{c}\right)P_2\left(\frac{z}{c+1}\right)P_3\left(\frac{z}{c-1}\right),\notag
	\\
\end{align}
where, for each $j\in\{1,2,3\}$, $P_j$ denotes the product
\begin{align*}
	P_j(z)=z\prod_{n\in\mathbb{Z}^*}\left(1+\frac{z}{i\overline{\nu}_n^j}\right).
\end{align*}
Also, we introduce the notation
\begin{equation}\label{eq:cj}
	c_j=\begin{cases}
		c, & \mbox{ if }j=1
		\\
		c+1, & \mbox{ if }j=2
		\\
		c-1, & \mbox{ if }j=3.
	\end{cases}
\end{equation}
From \eqref{eq:newp} it follows that
\begin{equation}\label{eq:pcj}
	P(z)=\prod_{j\in\{1,2,3\}}c_jP_j\left(\frac{z}{c_j}\right).
\end{equation}
Since $\overline{\nu}_n^j=\nu_{-n}^j$ we have that $P_j$ can be written as
\begin{align*}
	P_j(z)=z \prod_{n\in\mathbb{N}^*}\left(1-iz\left(\frac{1}{\overline{\nu}_n^j}+\frac{1}{\overline{\nu}_{-n}^j}\right)
	-z^2\frac{1}{\overline{\nu}_n^j\overline{\nu}_{-n}^j}\right) =z\prod_{n\in\mathbb{N}^*}\left(1-iz\frac{2\Re(\nu_n^j)}{|\nu_n^j|^2}-z^2\frac{1}{|\nu_{-n}^j|^2}\right).	
\end{align*}

Since both $\displaystyle\sum_{n\in\mathbb{N}^*}\frac{\Re(\nu_n^j)}{|\nu_n^j|^2}$ and $\displaystyle\sum_{n\in\mathbb{N}^*}\frac{1}{|\nu_n^j|^2}$
are absolutely convergent series, we obtain that the  product $P_j$ is absolutely convergent. Thus, $P$ is an absolute convergent product and the rearrangement of its terms as in \eqref{eq:newp} is allowed.

According to \cite[Corollary 3.4]{martin2013null}, it follows that $P_j$ are sine type functions. Each of these functions is entire and of exponential type $\pi$. Hence, the product $P$ is an entire function of exponential type $\frac{\pi}{|c|}+\frac{\pi}{|1+c|}+\frac{\pi}{|1-c|}$. Moreover, given any $\delta>0$, $P$ will be bounded on the strip $|\Im(z)|<\delta$ by a positive constant $C_\delta$ depending only on $\delta$ and $c$ such that \eqref{eq:preal} is verified.
Furthermore, since according to \eqref{eq:red} we have that 
\begin{align*}
	\Im(i\nu_m^k)|\leq \max\left\{\left|\frac{M}{c}\right|,\,\, \left|\frac{M}{2(c+1)}\right|,\,\, \left|\frac{M}{2(c-1)}\right| \right\}:=M_c,\qquad (m,k)\in S,
\end{align*}
from Schwarz's Lemma applied to the function
\begin{align*}
	G(z)=\frac{1}{C_{M_c+1}}P(z-i\overline{\lambda}_m^k),
\end{align*}
and \eqref{eq:preal} for $\delta=M_c+1$, we deduce that
\begin{align*}
	|P(z)|\leq C_{M_c+1}\left|z+i\overline{\lambda}_m^k\right|,\qquad \left|z+i\overline{\lambda}_m^k\right|<1.
\end{align*}
The last estimate implies that
\begin{equation}\label{eq:x1}
	\left|\frac{P(x)}{x+i\overline{\lambda}_m^k}\right|\leq \frac{2C_{M_c+1}}{1+\left|x+\Im(\lambda_m^k)\right|}, \qquad \left|x+\Im(\lambda_m^k)\right|<1.
\end{equation}
Since for $\left|x+\Im(\lambda_m^k)\right|\geq 1$ we deduce from \eqref{eq:preal} with $\delta=1$ that
\begin{equation}\label{eq:x2}
	\left|\frac{P(x)}{x+i\overline{\lambda}_m^k}\right|\leq \frac{2C_{1}}{1+\left|x+\Im(\lambda_m^k)\right|},
\end{equation}
it follows from \eqref{eq:x1} and \eqref{eq:x2} that \eqref{eq:preal} holds.

To prove \eqref{eq:der}, let us recall that, since $P_j$ is a sine type function, there exists a constant $m_j^1$ such that (see, for instance, \cite[Corollary 1, Section 5, Ch.4]{young2001introduction})
\begin{equation}\label{eq:pder}
	P_j'(-i\overline{\nu}_m^j)\geq m_j^1>0,\qquad m\in \mathbb{Z}^*.
\end{equation}

Moreover, for each $\varepsilon>0$, there exists $m^2_j(\varepsilon)>0$ such that (see, for instance, \cite[Chapter 4, Section 5, Lemma 2]{young2001introduction})
\begin{equation}\label{eq:pp}
	P_j(z)\geq m^2_j(\varepsilon)e^{\pi \Im(z)},\qquad z\in\mathbb{C},\,\,\inf_{n\in\mathbb{Z}^*}\left|z+i\overline{\nu}_n^j\right|>\varepsilon.
\end{equation}
By taking into account \eqref{eq:pder}, we have that
\begin{align}\label{eq:der00}
	\left|P'(-i\overline{\lambda}_m^j)\right|&=\left|P_j'(-i\overline{\nu}_m^j)
\prod_{\substack{l\in\{1,2,3\}\\l\neq j}}c_lP_l\left(-i\frac{\overline{\lambda}_m^j}{c_l}\right)\right| \geq \min_{1\leq j\leq 3}\left\{m_j^1\right\}\,\prod_{\substack{l\in\{1,2,3\}\\l\neq j}}\left|c_lP_l\left(-i\frac{\overline{\lambda}_m^j}{c_l}\right)\right|.
\end{align}

Hence, in order to prove \eqref{eq:der} it remains to evaluate the quantities $P_l\left(-i\frac{\overline{\lambda}_m^j}{c_l}\right)$ with $l\neq j$. Firstly we consider the case $l=1$. Since from \eqref{eq:diff123} in Lemma \ref{lemma:dist1} we deduce that
\begin{align*}
	\left|-i\frac{\overline{\lambda}_m^j}{c_1}+i\overline{\nu}_n^1\right|\geq \frac{|M|}{c_1(1+M^2)}>0,\qquad m,n\in\mathbb{Z}^*,\,\,
	j\in\{2,3\},
\end{align*}
from \eqref{eq:pp} we deduce that
\begin{equation}\label{eq:p1}
	\left|P_1\left(-i\frac{\overline{\lambda}_m^j}{c_1}\right)\right|\geq m_1^2\left(\frac{|M|}{c_1(1+M^2)}\right)e^{-\pi |M/c_1|},\qquad
m\in\mathbb{Z}^*,\,\, j\in\{2,3\}.
\end{equation}
The same argument allows us to deduce that
\begin{equation}\label{eq:p23}
	\left|P_j\left(-i\frac{\overline{\lambda}_m^1}{c_j}\right)\right|\geq m_j^2\left(\frac{|M|}{c_1(1+M^2)}\right)e^{-\pi |M/c_1|},\qquad
m\in\mathbb{Z}^*,\,\, j\in\{2,3\}.
\end{equation}

The evaluation of the remaining quantities, $P_2\left(-i\frac{\overline{\lambda}_m^3}{c_2}\right)$ and $P_3\left(-i\frac{\overline{\lambda}_m^2}{c_3}\right)$, is more difficult, since, according to Lemma \ref{lemma.dist23}, the eigenvalues $\lambda_m^2$ and $\lambda_m^3$ can be very close one to another. We evaluate
$P_3\left(-i\frac{\overline{\lambda}_m^2}{c_3}\right)$ for $c>1$ and $m>0$ only, the other quantities being treated similarly.

Firstly, we notice that Lemma \ref{lemma:sir} ensures that $|\lambda_m^2|\geq \frac{|M|}{M^2+1}\neq 0$ while Lemmas
\ref{lemma:dist20} and \ref{lemma.dist23} imply that $\lambda_m^2\notin \left\{\lambda_n^3,\,\lambda_{-n}^2\,\,:\,\, n\in\mathbb{N}^*\right\}.$ Therefore, we have that
\begin{equation}\label{eq:inmmic}
	\inf_{1\leq m\leq N_\epsilon}\left|P_3\left(-i\frac{\overline{\lambda}_m^2}{c_3}\right)\right|>0,
\end{equation}
where $N_\epsilon$ is given by Lemma \ref{lemma.dist23}. Hence, it remains to evaluate $P_3\left(-i\frac{\overline{\lambda}_m^2}{c_3}\right)$ for
$m>N_\epsilon$. According to Lemma \ref{lemma.dist23} there exists $n_m>m$ such that 
\begin{equation}\label{eq.d3m}
	\frac{\gamma'}{|m|^2} \leq \left|\lambda_m^2-\lambda_{n_m}^3\right|.
\end{equation} 
We define the complex number $\varrho_m$ as follows
\begin{equation}\label{eq.vrrm}
	\varrho_m=\lambda_{m}^2 -i\frac{c-1}{2},
\end{equation}
and we remark that $\Im(\varrho_m)<\Im(\lambda_m^2)$. From Lemma \ref{lemma:inc} and \ref{lemma.dist23} we deduce that 
\begin{equation}\label{eq:oil}
	\begin{cases} 
		\displaystyle \left|\lambda_{-n}^2-\varrho_m\right| \leq \left|\lambda_{-n}^2-\lambda_{m}^2\right|, & n\geq 1,
		\\[5pt]
		\displaystyle  \left|\lambda_{n}^3-\varrho_m\right| \leq \left|\lambda_{n}^3-\lambda_{m}^2\right|, & 1\leq n< n_m,
		\\[5pt]
		\displaystyle  \left|\lambda_{n-1}^3-\varrho_m\right| \leq \left|\lambda_{n}^3-\lambda_{m}^2\right|, & n>n_m+1.
	\end{cases}
\end{equation} 
Inequalities \eqref{eq:oil} imply that there exists $C>0$ such that
\begin{align*}
\left|P_3\left(-i\frac{\overline{\lambda}_m^2}{c_3}\right)\right| &= \left|-i\frac{\overline{\lambda}_m^2}{c-1}\right|\prod_{n\in 		\mathbb{N}^*}\left|\left(1-\frac{\lambda_m^2}{\lambda_n^3}\right) \left(1-\frac{\lambda_m^2}{\lambda_{-n}^2}\right)\right|
\\
&\geq \left|\frac{\lambda_m^2}{c-1} \right|\,\left|1-\frac{\lambda_m^2}{\lambda_{n_m}^3}\right|\,\left|1-\frac{\lambda_m^2}{\lambda_{n_m+1}^3}\right| \,\prod_{n\in \mathbb{N}^*}\left|1-\frac{\varrho_m}{\lambda_{-n}^2}\right| \,\prod_{1\leq n\leq n_m-1}\left|1-\frac{\varrho_m}{\lambda_n^3}\right|\, \prod_{n\geq n_m+2}\left|1-\frac{\varrho_m}{\lambda_{n}^3}\right|
\\
&\geq C \left|\varrho_m \right|\, \left|\lambda_{n_m}^2-\lambda_m^3\right| \,\prod_{n\in \mathbb{N}^*}\left|1-\frac{\varrho_m}{\lambda_{-n}^2}\right| \,
\prod_{n\in\mathbb{N}^*}\left|1-\frac{\varrho_m}{\lambda_n^3}\right| 
\\
&=  C \, \left|\lambda_{n_m}^2- \lambda_m^3\right| \, \left|P_3\left(-i\frac{\overline{\varrho}_m}{c_3}\right)\right|.
\end{align*}
Since we have that there exists $\varepsilon>0$ such that
\begin{align*}
	\left|-i\frac{\overline{\varrho}_m}{c_3}+i\overline{\nu}_n^3\right|\geq \varepsilon,\qquad n\in\mathbb{Z}^*,
\end{align*}
from \eqref{eq:pp} we deduce that
\begin{equation}\label{eq:p23new}
	\left|P_3\left(-i\frac{\overline{\lambda}_m^2}{c_3}\right)\right|\geq C \, m_2^2\left( \varepsilon\right) e^{-\pi |M|/c_3} \left|\lambda_{n_m}^2- \lambda_m^3\right|,\qquad m\in\mathbb{Z}^*.
\end{equation}
Estimates \eqref{eq:inmmic}, \eqref{eq:p23new} and \eqref{eq:di3} imply that there exists $C>0$ such that
\begin{equation}\label{eq:der023}
	\left|P_3\left(-i\frac{\overline{\lambda}_m^2}{c_3}\right)\right|\geq \frac{C}{m^2}, \qquad m\in\mathbb{Z}^*.
\end{equation}
Analogously, we deduce that
\begin{equation}\label{eq:der023new}
	\left|P_2\left(-i\frac{\overline{\lambda}_m^3}{c_2}\right)\right|\geq \frac{C}{m^2}, \qquad m\in\mathbb{Z}^*.
\end{equation}

By using \eqref{eq:p1}, \eqref{eq:p23}, \eqref{eq:der023} and \eqref{eq:der023new}, from \eqref{eq:der00} it follows that \eqref{eq:der} holds true and the proof of the Theorem is complete.
\end{proof}
\end{theorem}

\noindent The previous theorem allows us to construct the biorthogonal family we were looking for.

\begin{theorem}\label{te:bio} 
Let $c\in \mathbb{R}\setminus\{-1,0,1\}$ and $T>2\pi\left( \frac{1}{|c|}+\frac{1}{|c-1|}+\frac{1}{|c+1|}\right)$. There exist a biorthogonal sequence $\left(\theta_m^k\right)_{(m,k)\in S}$ to the family of complex
exponentials $\left(e^{-\lambda_n^j t}\right)_{(n,j)\in S}$ in $L^2\left(-\frac{T}{2},\frac{T}{2}\right)$ and a positive constant $C$ with the property that
\begin{equation}\label{eq:biono}
	\left\|\sum_{(m,k)\in S}\beta_m^k\theta_m^k\right\|^2_{L^2\left(-\frac{T}{2},\frac{T}{2}\right)}\leq C \sum_{(m,k)\in S}m^4\left|\beta_m^k\right|^2,
\end{equation}
for any finite sequence of complex numbers $\left(\beta_m^k\right)_{(m,k)\in S}$. 
\end{theorem}

\begin{proof}
We define the entire function
\begin{equation}\label{eq:wht}
	\widehat{\theta}_m^j(z)=\displaystyle \frac{P(z)}{P'(-i\overline{\lambda}_m^j)},
\end{equation} 
and let
\begin{equation}\label{eq:bt}
	\theta_m^j=\displaystyle \frac{1}{2\pi}\int_\mathbb{R} \widehat{\theta}_m^j(x)e^{ixt}\,{\rm dx}.
\end{equation} 

From Theorem \ref{te:pprod} we deduce that $(\theta_m^j)_{(m,j)\in S}$ is a biorthogonal sequence to the family of exponential functions $\Lambda=\left(e^{-\lambda_{n}^jt}\right)_{(n,j)\in S}$ in $L^2\left(-\frac{T'}{2},\frac{T'}{2}\right)$, where  $T'= 2\pi\left( \frac{1}{|c|}+\frac{1}{|c-1|}+\frac{1}{|c+1|}\right)$. Moreover, we have that
\begin{align*}
	\|\theta_m^j\|_{L^2\left(-\frac{T'}{2},\frac{T'}{2}\right)}\leq C\, m^2 \qquad ((m,j)\in S).
\end{align*}

An argument similar to \cite{kahane1962pseudo} (see, also, \cite[Proposition 8.3.9]{tucsnak2009observation}) allows us to prove that, for any $T>T'$, there exists a biorthogonal sequence  $(\theta_m^j)_{(m,j)\in S}$ to  the family of exponential functions $\Lambda=\left(e^{-\lambda_{n}^jt}\right)_{(n,j)\in S}$ in
$L^2\left(-\frac{T}{2},\frac{T}{2}\right)$ such that \eqref{eq:biono} is verified.
\end{proof}

The following immediate consequence of Theorem \ref{te:bio} will be very useful for the controllability problem studied in the next section. 

\begin{corollary}
Let $c\in \mathbb{R}\setminus\{-1,0,1\}$ and $T>2\pi\left( \frac{1}{|c|}+\frac{1}{|c-1|}+\frac{1}{|c+1|}\right)$.  For any finite sequence of scalars $(a_n^j)_{(n,j)\in S}\subset \mathbb{C}$, it holds the inequality
\begin{align}\label{eq:insum}
	\sum_{(n,j)\in S}\frac{\left|a_n^j\right|^2 }{n^4}\leq  C \norm{\sum_{(n,j)\in S}a_{n}^j e^{-\lambda_n^j t}}{L^2\left(-\frac{T}{2},\frac{T}{2}\right)}^2,
\end{align}
\end{corollary}
\begin{proof} 
By taking into account the orthogonality properties of $\left(\theta_m^k\right)_{(m,k)\in S}$ we deduce that
\begin{align*}
	\sum_{(n,j)\in S}\frac{\left|a_n^j\right|^2}{n^4} &= \int_{-\frac{T}{2}}^{\frac{T}{2}}
	\overline{\left( \sum_{(n,j)\in S} a_n^j e^{-\lambda_n^j t}\right)}\left( \sum_{(m,k)\in S} \frac{a_m^k }{m^4}\theta_m^k(t)\right)\,{\rm dt} 
	\\
	&\leq \norm{\sum_{(n,j)\in S}a_{n}^j e^{-\lambda_n^j t}}{L^2\left(-\frac{T}{2},\frac{T}{2}\right)}\, \norm{\sum_{(m,k)\in S}\frac{a_{m}^k}{m^4}\theta_m^k}{L^2\left(-\frac{T}{2},\frac{T}{2}\right)},
\end{align*}
from which, by taking into account \eqref{eq:biono}, we easily deduce \eqref{eq:insum}. 
\end{proof}

\section{Controllability results}\label{control_sect}

In this section we study the controllability properties of equation \eqref{wave_mem}, and we prove that it is memory-type null-controllable if the initial data $(y^0,y^1)$ are smooth enough.  We start by reducing our original problem to a moment problem which, in a second moment, we will solve with the help of the biorthogonal sequence that we constructed in Section \ref{bio_sec}. Let us begin with the following result concerning the solutions of \eqref{wave_mem_adj_syst_cv}.

\begin{lemma}\label{sol_adj_lemma}
For each initial data 
\begin{align}\label{in_data_adj}
	\left(\begin{array}{c}
		\varphi(T,x)\\ \varphi_t(T,x)\\ \psi(T,x)
	\end{array}\right) = \left(\begin{array}{c}
	\varphi^0(x)\\ \varphi^1(x)\\ \psi^0(x)
	\end{array}\right) = \sum_{(n,j)\in S} b_n^{\,j}\Psi_n^j(x) \in \Lp\times\Hs{-1}\times\Lp, 
\end{align} 
there exists a unique solution of equation \eqref{wave_mem_adj_syst_cv} given by  	
\begin{align}\label{sol_adj}
	\left(\begin{array}{c} \varphi(t,x)\\ \varphi_t(t,x)\\ \psi(t,x)\end{array}\right) = \sum_{(n,j)\in S} b_n^{\,j}e^{\lambda_n^j(T-t)}\Psi_n^j(x).
\end{align}
\end{lemma}

\begin{proof}
Firstly, let us notice that, since according to Theorem \ref{te:lari}, $(\Psi_n^j)_{(n,j)\in S}$ is a Riesz basis, each initial data in $ \Lp\times\Hs{-1}\times\Lp$ can be given in the form \eqref{in_data_adj} with $(b_n^{\,j})_{(n,j)\in S}\in\ell^2$. Then, the proof is finished if we remark that, for $(n,j)\in S$, if we consider as initial data $\Psi_n^j$, the solution to \eqref{wave_mem_adj_syst_cv} is given by 
\begin{align}\label{sol_adj_expr}
	\left(\begin{array}{c} \varphi(t,x)\\ \varphi_t(t,x)\\ \psi(t,x)\end{array}\right) = e^{\lambda_n^j(T-t)}\Psi_n^j(x).
\end{align}
\end{proof}

\noindent We have the following result which reduces the controllability problem to a problem of moments.
\begin{lemma}\label{moment_lemma}
Let $0\leq\sigma<\infty$. The equation \eqref{wave_mem} is memory-type null controllable at time $T$ if, for each $(y^0,y^1)\in\Hs{\sg+1}\times\Hs{\sg}$,  
\begin{align}\label{data_dec}
	y^0(x) = \sum_{n\in\ZZ^*} y_0^ne^{inx},\quad y^1(x) = \sum_{n\in\ZZ^*} y_1^ne^{inx},
\end{align}
there exists $\widehat{u}\in L^2(Q)$ such that the following relations hold
\begin{align}\label{moment_pb1}
	\int_0^T\int_{\omega_0} \widehat{u}(t,x)e^{-inx}e^{-\bar{\lambda}_n^jt}\,{\rm dx\,dt} = -2\pi\Big(\bar{\mu}_{|n|}^jy_n^0 + y_n^1\Big),\qquad (n,j)\in S,\\
	\int_0^T\int_{\omega_0} \widehat{u}(t,x) e^{-\bar{\lambda}_n^jt}\,{\rm dx\,dt} = 0,\qquad (n,j)\in S.\label{moment_pb2}
\end{align} 
\end{lemma}
\begin{proof} 

First of all, recall that, according to Lemma \ref{control_id_lemma_syst}, equation \eqref{wave_mem} is memory-type null controllable in time $T$ if and only if, for each $(y^0,y^1)\in\Hs{\sg+1}\times\Hs{\sg}$,  there exists $\tilde{u}\in L^2(Q)$ such that \begin{equation}\label{eq:utm0} \int_{-\pi}^\pi  \mathbf{1}_{\omega_0} \tilde{u}(t,x)\,{\rm dx}=0, \qquad t\in (0,T),\end{equation} and the following identity holds 
\begin{align}\label{control_id_cv_2}
	\int_0^T\int_{\omega_0} \tilde{u}(t,x)\bar{\varphi}(t,x)\,{\rm dx\,dt} = \big\langle y^0(\cdot),\varphi_t(0,\cdot)\big\rangle_{\Hs{1},\Hs{-1}} - \int_{\mathbb{T}} (y^1(x)+cy^0_x(x))\bar{\varphi}(0,x)\,{\rm dx},
\end{align}
where, for all $(p^0,p^1,q^0)\in\Lp\times\Hs{-1}\times\Lp$, $(\varphi,\psi)$ is the unique solution to \eqref{wave_mem_adj_syst_cv}. In fact, to verify \eqref{control_id_cv_2} it is sufficient to consider as initial data the elements of the Riesz basis $\left\{\Psi_n^j\right\}_{(n,j)\in S}:$
\begin{align*}
	\left(\begin{array}{c} p^0\\[5pt]p^1-cp_x^0\\[5pt]q^0\end{array}\right) = \Psi_n^j = \left(\begin{array}{c} 1\\[5pt]-\lambda_n^j\\[5pt] \displaystyle \frac{1}{\lambda_n^j-icn}\end{array}\right)e^{inx}.
\end{align*}

Then, according to Lemma \ref{sol_adj_lemma}, the solution to \eqref{wave_mem_adj_syst_cv} can be written in the form \eqref{sol_adj_expr}. Moreover, we can readily check that 
\begin{align*}
    \int_{\mathbb{T}}\big(y^1+cy^0_x\big)\bar{\varphi}(0,x)\,{\rm dx} = \sum_{(n,j)\in S}\big(y_n^1+icn y_n^0\big)e^{\bar{\lambda}_m^jT}\int_{\mathbb{T}} e^{inx}e^{-imx}\,{\rm dx} = 2\pi\big(y_n^1+icn y_n^0\big)e^{\bar{\lambda}_n^jT},
\end{align*}
where we used the orthogonality of the eigenfunctions $e^{inx}$ in $\Lp$. In a similar way, we also have
\begin{align*}
    \big\langle y^0,\varphi_t(0,\cdot)\big\rangle_{\Hs{1},\Hs{-1}} = -\sum_{(n,j)\in S}y_n^0\bar{\lambda}_n^je^{\bar{\lambda}_n^jT},
\end{align*}
and from \eqref{control_id_cv_2} we finally obtain that \eqref{control_id_cv_2} is equivalent to 
\begin{align}\label{moment_pb3}
	\int_0^T\int_{\omega_0} \tilde{u}(t,x)e^{-inx}e^{-\bar{\lambda}_n^jt}\,{\rm dx\,dt} = -2\pi\Big(\bar{\mu}_{|n|}^jy_n^0 + y_n^1\Big),\qquad (n,j)\in S.
\end{align} 
Now, let $\widehat{u}\in L^2(Q)$ verifying \eqref{moment_pb1}-\eqref{moment_pb2} and let us define 
\begin{align*}
	\tilde{u}(t,x)=\widehat{u}(t,x)-\frac{1}{|\omega_0|}\int_{\omega_0}\widehat{u}(t,s)\,{\rm ds},\qquad (t,x)\in Q.
\end{align*}

Clearly, $\tilde{u}\in L^2(Q)$  and \eqref{eq:utm0} is verified. It remains to show that $\tilde{u}$ also satisfies \eqref{moment_pb3}. Indeed, by taking into account  the definition of $\tilde{u}$ and relations \eqref{moment_pb1}-\eqref{moment_pb2}, we deduce that 
\begin{align*}	
	\int_0^T\int_{\omega_0} & \tilde{u}(t,x)e^{-inx}e^{-\bar{\lambda}_n^jt}\,{\rm dx\,dt} 
	\\
	&=\int_0^T\int_{\omega_0} \widehat{u}(t,x)e^{-inx}e^{-\bar{\lambda}_n^jt}\,{\rm dx\,dt} - \frac{1}{|\omega_0|}\int_{\omega_0}e^{-inx}\,{\rm dx}\int_0^T\int_{\omega_0} \tilde{u}(t,x)e^{-\bar{\lambda}_n^jt}\,{\rm dx\,dt}
	\\
	&=-2\pi\left(\bar{\mu}_{|n|}^jy_n^0 + y_n^1\right).
\end{align*}
Hence, $\tilde{u}$ is a control and the proof of the lemma is complete.
\end{proof} 
In order to solve the moment problem \eqref{moment_pb1}-\eqref{moment_pb2}, we shall use the following result (see \cite[Chapter 4, Section 1, Theorem 2]{young2001introduction}).
\begin{theorem}\label{thm_young}
Let $(f_n)_n$ be a sequence of vectors belonging to a Hilbert space $H$ and $(c_n)_n$ a sequence of scalars. In order that the equations
\begin{align*}
	(f,f_n)=c_n
\end{align*}
shall admit at least one solution $f\in H$ for which $\norm{f}{H}\leq M$, it is necessary and sufficient that 
\begin{align}\label{rel_young}
	\left|\sum_n a_n\bar{c}_n\right|\leq M\norm{\sum_n a_nf_n}{H}
\end{align}
for every finite sequence of scalars $(a_n)_n$.
\end{theorem} 

\noindent We can now pass to prove the main controllability result.

\begin{proof}[Proof of Theorem \ref{control_thm}]
According to Lemma \ref{moment_lemma}, it is sufficient to show that, for each initial data $(y^0,y^1)\in\Hs{3}\times\Hs{2}$ given by \eqref{data_dec}, there exists $\widehat{u}\in L^2(Q)$ such that \eqref{moment_pb1}-\eqref{moment_pb2} are satisfied. From Theorem \ref{thm_young}, this is equivalent to show that the following inequality holds
\begin{align}\label{young_ineq}
	\left| \sum_{(n,j)\in S} (\mu_{|n|}^j\bar{y}_n^0+\bar{y}_n^1)a_n^j\right|^2 \leq C\int_0^T\int_{\omega_0} \left|\sum_{(n,j)\in S} b_n^j e^{-\lambda_n^jt}+\sum_{(n,j)\in S} a_n^je^{inx}e^{-\lambda_n^jt}\right|^2\,{\rm dx\,dt},
\end{align}
for all finite sequences $(a_n^j)_{(n,j)\in S}\cup (b_n^j)_{(n,j)\in S}\subset \mathbb{C}$. To this end, we notice that
\begin{align}\label{ineq1}
	\left| \sum_{(n,j)\in S} (\mu_{|n|}^j\bar{y}_n^0+\bar{y}_n^1)a_n^j\right|^2 &\leq \left(\sum_{(n,j)\in S} n^4\left|\mu_{|n|}^j\bar{y}_n^0+\bar{y}_n^1\right|^2\right)\left(\sum_{(n,j)\in S} \frac{|a_n^j|^2}{n^4}\right) \notag 
	\\
	&\leq C\norm{(y^0,y^1)}{\Hs{3}\times\Hs{2}}^2 \left(\sum_{(n,j)\in S} \frac{|a_n^j|^2}{n^4}\right).
\end{align}

On the other hand, by using \eqref{eq:insum} and taking into account that we can have at most one double eigenvalue $\lambda^2_{-n_c}$ (if $c\in {\mathcal V}$, see Lemma \ref{lemma:dist20}), we deduce that
\begin{align*}
	\int_0^T\int_{\omega_0} &\left|\sum_{(n,j)\in S} b_n^j e^{-\lambda_n^jt}+ \sum_{(n,j)\in S} a_n^j e^{inx}  e^{-\lambda_n^jt}\right|^2\,{\rm dx\,dt} 
	\\
	=&\; \int_{\omega_0}\int_{-\frac{T}{2}}^{\frac{T}{2}}\left| \sum_{(n,j)\in S}\left( a_n^je^{inx} +b_n^j\right)e^{\lambda_n^j \frac{T}{2}}   e^{-\lambda_n^j t}\right|^2\,{\rm dt\,dx}
	\\
	\geq&\; C \Bigg(
	\sum_{(n,j)\in S\setminus\{(-n_c,2),\,(n_c,3)\}} \frac{1}{n^4}  \int_{\omega_0} \left|\left(a_n^je^{inx}+b_n^j\right) e^{\lambda_n^j \frac{T}{2}} \right|^2 \,{\rm dx}
	\\
	&+\displaystyle \int_{\omega_0} \left|a_{-n_c}^2 e^{\lambda_{-n_c}^2 \frac{T}{2}} e^{-in_cx} +  a_{n_c}^3 e^{\lambda_{n_c}^3 \frac{T}{2}} e^{in_cx}+b_{-n_c}^2 e^{\lambda_{-n_c}^2 \frac{T}{2}} + b_{n_c}^3 e^{\lambda_{n_c}^3 \frac{T}{2}} \right|^2 \,{\rm dx}\Bigg).
\end{align*} 

Let us mention that, if $c\notin{\mathcal V}$, all the eigenvalues are simple and the separation of the second term in the last inequality is not needed. Since the maps
\begin{align*}
	&\CC^2\ni(a,b)\mapsto \left(\int_{\omega_0}\left|ae^{inx} + b\right|^2\,{\rm dx}\right)^{\frac 12},\\ \\
	&\CC^3\ni(a',a'',b)\mapsto \left(\int_{\omega_0}\left|a'e^{-in_cx}  + a''e^{in_cx} +b\right|^2\,{\rm dx}\right)^{\frac 12},
\end{align*}
are norms in $\CC^2$ and $\CC^3$, respectively, it follows that 
\begin{align}\label{ineq2}
	\int_0^T\int_{\omega_0}\left|\sum_{(n,j)\in S} b_n^j e^{-\lambda_n^jt}+ \sum_{(n,j)\in S} a_n^j e^{inx}  e^{-\lambda_n^jt}\right|^2\,{\rm dx\,dt}  \geq C \sum_{(n,j)\in S}\frac{|a_n^j|^2}{n^4}.
\end{align}
From \eqref{ineq1} and \eqref{ineq2} we immediately obtain \eqref{young_ineq}. Our proof is then concluded.
\end{proof}

\begin{remark}
We have obtained that \eqref{wave_mem} is controllable if the control time is larger than
\begin{align}\label{eq:t00}
	T_0=2\pi\left(\frac{1}{|c|}+\frac{1}{|1-c|}+\frac{1}{|1+c|}\right).
\end{align}

This result is a consequence of the particular construction on the biorthogonal sequence in Theorem \ref{te:bio}. To determine the optimal control time remains an interesting open problem. In particular, notice that the minimal control time should depend on the set $\omega_0$. 
\
\
However, we can show that, in the limiting case in which the support of the control is reduced to one point (or it is of the form $u(t,x)=b(x+ct)v(t)$ with $b$ given), the minimal control time is precisely $T_0$.  Indeed, in this case, the controllability on time $T$ property implies the existence of a nontrivial  entire function  $G$  such that 
\begin{itemize}
	\item $G(-\overline{\lambda}_n^j)=0$ for each $(n,j)\in S$;
	\item $G$ is of exponential type $T$;
	\item $G$ belongs to $L^2$ on the real axis.
\end{itemize}
Under these hypotheses, $G$ can be factorized as 
\begin{align*}
	G(z)=\frac{1}{z^3}P(z)M(z),
\end{align*} 
where $P$ is the infinite product defined by \eqref{eq:prod} and $M$ is an entire function. According to Theorem \ref{te:pprod}, $P$ is an entire function of exponential type $T_0$. Moreover, from the definition of a sine-type function, we have that 
\begin{equation}\label{eq:ppp1}
	|P(z)|\geq C \exp\left( T_0\,  |\Im\, z|\right),
\end{equation}
for any $z\in\mathbb{C}$ with $|\Im\, z|$ larger than a constant $H>0$ and \eqref{eq:pp} ensures the existence of a constant $C>0$ and a sequence of positive numbers $(r_n)_{n\geq 1}$ such that $\lim_{n\rightarrow \infty }r_n=\infty$ and 
\begin{equation}
	|P(z)|\geq C, \qquad z\in \left\{z\in\mathbb{C}\,:\, |\Im\, z|\leq H,\,\, |z|=r_n,\,\, n\geq 1\right\}.
\end{equation}
On the other hand, the properties of $G$ and the Paley-Wiener theorem imply that 
\begin{equation}\label{eq:pppp2}
	|G(z)|\leq C, \exp\left( T \, |\Im\, z|\right)\qquad z\in\mathbb{C}.
\end{equation}
Let us suppose that $T\leq T_0$. From \eqref{eq:ppp1}-\eqref{eq:pppp2} it follows that 
\begin{equation}
	|M(z)|\leq C |z|^3, \qquad z\in\mathbb{C},\,\, | z |=r_n,\,\, n\geq 1.
\end{equation}

By using the Cauchy's differentiation formula we can deduce that $M$ is a polynomial function and, consequently, the exponential type $T$ of $G$ has to be equal to $T_0$. Hence, we have shown that $T\geq T_0$ and, in this context, $T_0$ given by \eqref{eq:t00} does represent the minimal control time. 
\end{remark}

\begin{remark} Let us mention that the space of controllable initial data is larger than the one given by Theorem \ref{control_thm}. This is a consequence of the fact that  the small weight $\frac{1}{n^4}$ in inequality  \eqref{ineq2} affects only some terms in the right hand side series. However, it is not easy to identify a larger classical space of controllable initial data than $\Hs{3}\times\Hs{2}$.
\end{remark}

\begin{remark}
We remark that our method for proving Theorem \ref{control_thm} is based on the problem of moments \eqref{moment_pb1}-\eqref{moment_pb2} and the construction of a biorthogonal sequence in Theorem \ref{te:bio}. This approach does not provide an explicit control, and this may become relevant when facing practical applications and numerical issues. On the other hand, we have to mention that, in what concerns the numerical approximation of our control problem, this explicit construction is not necessary. As it is often done in control theory, the function $u$ may be computed through the minimization of a suitable functional associated to our problem. For doing that, we need an accurate discretization of \eqref{wave_mem} and of the adjoint \eqref{wave_mem_adj}. Although this may appear a tricky task due to the presence of an integral term in the equations, this problem may be overcome by working with the equivalent formulations \eqref{wave_mem_syst} and \eqref{wave_mem_adj_syst}. In this way, we are reduced to the discretization of coupled PDE/ODE systems, which may be easily performed by means of classical techniques (see, e.g., \cite{colli2004parallel}). On the other hand, since we are dealing with a wave-type equation, most likely we will have to overcome the difficulties caused by the high-frequency spurious numerical solutions for which the group velocity vanishes and a filtering mechanism will be probably necessary. The efficient numerical approximation of the controls for this system remains an interesting problem which needs further investigation.
\end{remark}

\section{Construction of a localized solution}\label{loc_sol_sect}

The aim of this section is to provide a more physical justification for the lack of controllability of equation \eqref{wave_mem} in the case of a control which is not moving. For doing that, we are going to show that it is possible to construct a quasi-solution to the equation which is localized around a vertical characteristic and which, if the control is not moving, cannot be observed. In what follows, we will always consider $x\in\RR$.

Let us start by computing the characteristics to our equation. With this purpose, we take one time derivative of \eqref{wave_mem}, and we find
\begin{align}\label{wave_mem_der}
	y_{ttt}-y_{xxt}+My_{xx}=0.
\end{align}

The principal part of the above operator is given by the two third order terms, and its symbol in the Fourier variables $(\tau,\xi)$ is
\begin{align*}
	\rho(\tau,\xi) = \tau^3-\tau|\xi|^2 = \tau(\tau^2-|\xi|^2).
\end{align*}

We therefore see that the equation, in addition to the classical characteristics of the wave equation, admits also a vertical characteristic which, for any time $t$, is not propagating in the space variables.

It is nowadays well-known that control properties for a given PDE may be obtained through the observation of its adjoint equation which, we recall, in the present case reads as
\begin{align}\label{wave_mem_adj_R}
	\begin{cases}
		\displaystyle p_{tt}(t,x)-p_{xx}(t,x) + M\int_0^t p_{xx}(s,x)\,{\rm ds} + M q_{xx}^0(x) = 0, & (t,x)\in (0,T)\times\RR
		\\
		p(0,x)=p^0(x),\;\;p_t(0,x)=p^1(x), & x\in\RR
	\end{cases}
\end{align}

In what follows we will show that it is possible to construct a highly oscillating quasi-solution of \eqref{wave_mem_adj_R}, which is localized around the vertical characteristic. This means that the information carried by this solution has no possibility to move toward the control and that, instead, it is the control that shall move toward it.

Taking inspiration from the classical theory of geometric optics (see, e.g., \cite{ralston1982gaussian,rauch2005polynomial}), given any $\varepsilon>0$ small we consider the following ansatz
\begin{align}\label{loc_sol_prel}
	\ue{p}(t,x):= c(\varepsilon)e^{-\frac{1}{\varepsilon^r}(x-x_0)^2}a(t,x), \;\;\; x_0\in\RR,
\end{align}
where $r>0$, $c(\varepsilon)$ is a normalization constant and the function $a(t,x)$ has to be determined. 

With this aim, we ask that $\ue{p}$ satisfies \eqref{wave_mem_adj} up to a small error. We recall that, up to a derivation in the time variable, \eqref{wave_mem_adj_R} is equivalent to the equation
\begin{align}\label{wave_mem_adj_der}
	p_{ttt}-p_{xxt} + Mp_{xx} = 0.
\end{align}
In addition, we can easily compute
\begin{align*}
	&\ue{p}_{ttt} - \ue{p}_{xxt} + M\ue{p}_{xx} 
	\\
	&\;\;= c(\varepsilon) e^{-\frac{1}{\varepsilon^r}(x-x_0)^2}\!\bigg[a_{ttt} - a_{xxt} +  Ma_{xx} + \frac{2}{\varepsilon^r}\Big(a_t-Ma + 2(x-x_0)(a_t-Ma)_x\Big) \!- \frac{4(x-x_0)^2}{\varepsilon^{2r}}(a_t-Ma)\bigg].
\end{align*}
Choosing now $a$ in the form
\begin{align*}
	a(t,x) = e^{\frac i\varepsilon x}e^{Mt-M^2\varepsilon^2t}.
\end{align*}
it is simple a matter of computations to obtain
\begin{align*}
	\ue{p}_{ttt} - \ue{p}_{xxt} + M\ue{p}_{xx} = c(\varepsilon)ae^{-\frac{1}{\varepsilon^r}(x-x_0)^2} \bigg[\mathcal O(\varepsilon^2) + \mathcal O(\varepsilon^{2-r}) + \mathcal O(\varepsilon^{1-r}) + \mathcal O(\varepsilon^{2-2r})\bigg] = c(\varepsilon)\mathcal O(\varepsilon^{1-r}).
\end{align*}

From the above identity we see that, if we want a good approximation of the solution to \eqref{wave_mem_adj_R}, we shall choose $r<1$. Moreover, apart form this limitation, the choice of $r$ may be arbitrary. Taking, e.g., $r=1/2$, we then have
\begin{align*}
	\ue{p}_{ttt} - \ue{p}_{xxt} + M\ue{p}_{xx} = c(\varepsilon)\mathcal O(\varepsilon^{\frac 12}).
\end{align*}
In view of that, our candidate for a localized solution will be the following:

\begin{align*}
	\ue{p}(t,x) = c(\varepsilon) e^{\frac i\varepsilon x - \frac{1}{\sqrt{\varepsilon}} (x-x_0)^2 + Mt - M^3\varepsilon^2t}.
\end{align*}
The normalization constant $c(\varepsilon)$ is chosen so that
\begin{align*}
	p^\varepsilon(0,x) := p^{\varepsilon,0}(x)= c(\varepsilon) e^{\frac i\varepsilon x - \frac{1}{\sqrt{\varepsilon}} (x-x_0)^2}
\end{align*}
has unitary $H^1(\RR)$-norm (up to some small reminder). To this end, let us compute
\begin{align*}
	\norm{p^{\varepsilon,0}}{H^1(\RR)} = \left(\norm{p^{\varepsilon,0}}{L^2(\RR)}^2 + \norm{p^{\varepsilon,0}_x}{L^2(\RR)}^2\right)^{\frac 12}.
\end{align*}
First of all, we have
\begin{align*}
	\norm{p^{\varepsilon,0}}{L^2(\RR)}^2 = c(\varepsilon)^2\int_\RR e^{- \frac{2}{\sqrt{\varepsilon}} (x-x_0)^2}\,{\rm dx} = c(\varepsilon)^2\frac{\varepsilon^{\frac 14}}{\sqrt{2}}\int_\RR e^{-z^2}\,{\rm dz} = c(\varepsilon)^2 \frac{\sqrt{\pi}}{\sqrt{2}}\,\varepsilon^{\frac 14}.
\end{align*}
Moreover, a simple computation gives
\begin{align*}
	p^{\varepsilon,0}_x = c(\varepsilon)e^{- \frac{1}{\sqrt{\varepsilon}} (x-x_0)^2}\bigg\{-\frac x\varepsilon\sin\left(\frac x\varepsilon\right) - \frac{(x-x_0)^2}{\sqrt{\varepsilon}} \cos\left(\frac x\varepsilon\right) +\left. i\left[\frac x\varepsilon\cos\left(\frac x\varepsilon\right) - \frac{(x-x_0)^2}{\sqrt{\varepsilon}} \sin\left(\frac x\varepsilon\right) \right]\right\}.
\end{align*}
Hence, we easily obtain
\begin{align*}
	\left|p^{\varepsilon,0}_x\right|^2 = \Big[\Re(p^{\varepsilon,0}_x)\Big]^2 + \Big[\Im(p^{\varepsilon,0}_x)\Big]^2 = c(\varepsilon)^2e^{- \frac{2}{\sqrt{\varepsilon}} (x-x_0)^2}\left[\frac 1\varepsilon (x-x_0)^4 + \frac{x^2}{\varepsilon^2}\right],
\end{align*}
and we get
\begin{align*}
	\int_\RR \left|p^{\varepsilon,0}_x\right|^2\,{\rm dx} = c(\varepsilon)^2\int_\RR \frac 1\varepsilon (x-x_0)^4e^{- \frac{2}{\sqrt{\varepsilon}} (x-x_0)^2}\,{\rm dx} + c(\varepsilon)^2\int_\RR \frac{x^2}{\varepsilon^2}e^{- \frac{2}{\sqrt{\varepsilon}} (x-x_0)^2}\,{\rm dx}.
\end{align*}
Let us firstly compute
\begin{align*}
	c(\varepsilon)^2\int_\RR \frac 1\varepsilon (x-x_0)^4e^{- \frac{2}{\sqrt{\varepsilon}} (x-x_0)^2}\,{\rm dx} = \frac{c(\varepsilon)^2\varepsilon^{\frac 14}}{4\sqrt{2}}\int_\RR z^4e^{-z^2}\,{\rm dz} = c(\varepsilon)^2\frac{3\sqrt{\pi}}{16\sqrt{2}}\varepsilon^{\frac 14}.
\end{align*}
Concerning now the second integral, we have
\begin{align*}
	c(\varepsilon)^2\int_\RR \frac{x^2}{\varepsilon^2} e^{- \frac{2}{\sqrt{\varepsilon}} (x-x_0)^2}\,{\rm dx} =&\, c(\varepsilon)^2\frac{\varepsilon^{-\frac 74}}{\sqrt{2}}\int_\RR \left(x_0 + \frac{\varepsilon^{\frac 14}}{\sqrt{2}}z\right)^2 e^{-z^2}\,{\rm dz}
	\\
	=&\, c(\varepsilon)^2\frac{x_0^2}{\sqrt{2}}\varepsilon^{-\frac 74}\int_\RR e^{-z^2}\,dz + c(\varepsilon)^2x_0\varepsilon^{-\frac 32}\int_\RR z e^{-z^2}\,dz 
	\\
	&+c(\varepsilon)^2+ \frac{\varepsilon^{-\frac 54}}{2\sqrt{2}}\int_\RR z^2 e^{-z^2}\,{\rm dz}
	\\
	=&\, c(\varepsilon)^2\frac{x_0^2\sqrt{\pi}}{\sqrt{2}}\varepsilon^{-\frac 74} + c(\varepsilon)^2\frac{\sqrt{\pi}}{4\sqrt{2}} \varepsilon^{-\frac 54}.
\end{align*}
Thus,
\begin{align*}
	\int_\RR |p^{\varepsilon,0}_x|^2\,{\rm dx} = c(\varepsilon)^2\left(\frac{3\sqrt{\pi}}{16\sqrt{2}}\varepsilon^{\frac 14} + \frac{x_0^2\sqrt{\pi}}{\sqrt{2}}\varepsilon^{-\frac 74} + \frac{\sqrt{\pi}}{4\sqrt{2}} \varepsilon^{-\frac 54}\right),
\end{align*}
and, adding all the contributions, we get
\begin{align*}
	\norm{p^{\varepsilon,0}}{H^1(\RR)} = c(\varepsilon)\left(\frac \pi2\right)^{\frac 14}\left(\frac{19}{6}\varepsilon^{\frac 14} + x_0^2\varepsilon^{-\frac 74} + \frac 14 \varepsilon^{-\frac 54}\right)^{\frac 12}.
\end{align*}
We now have to distinguish two cases. If $x_0\neq 0$, from the above computations we obtain
\begin{align*}
	\norm{p^{\varepsilon,0}}{H^1(\RR)} = c(\varepsilon)x_0\left(\frac \pi2\right)^{\frac 14}\varepsilon^{-\frac 78}\left(1 + \mathcal O(\varepsilon^{\frac 14})\right).
\end{align*}
On the other hand, if $x_0=0$ we instead have
\begin{align*}
	\norm{p^{\varepsilon,0}}{H^1(\RR)} = c(\varepsilon)\left(\frac{\pi}{32}\right)^{\frac 14}\varepsilon^{-\frac 58}\left(1 + \mathcal O(\varepsilon^{\frac 34})\right).
\end{align*}
Therefore, choosing
\begin{align}\label{norm_const}
c(\varepsilon) = \begin{cases}
\displaystyle \frac{1}{x_0}\left(\frac 2\pi\right)^{\frac 14}\varepsilon^{\frac 78}, & x_0\neq 0,
\\[10pt]
\displaystyle\left(\frac{32}{\pi}\right)^{\frac 14}\varepsilon^{\frac 58}, & x_0=0,
\end{cases}
\end{align}
we finally have that $\norm{p^{\varepsilon,0}}{H^1(\RR)} \to 1$  as $\varepsilon\to 0.$

The construction that we just presented provided us a candidate for an approximate solution to \eqref{wave_mem_adj_R}, which is concentrated along the vertical characteristic $x(t)=x_0$. We now need to rigorously prove this fact. In what follows, for the sake of simplicity we will always assume that $x_0\neq 0$. If $x_0=0$, the only difference will be in the normalization constant \eqref{norm_const}, but the results that we are going to present will still hold.

\begin{theorem}
    For any $x_0\in\RR^\ast=\RR\setminus\{0\}$ and $\varepsilon>0$, let the functions $\ue{p}$ be defined as
    \begin{align}\label{pe}
    \ue{p}(t,x):= \frac{1}{x_0}\left(\frac 2\pi\right)^{\frac 14}\varepsilon^{\frac 78} e^{\frac i\varepsilon x - \frac{1}{\sqrt{\varepsilon}} (x-x_0)^2 + Mt - M^3\varepsilon^2t}.
    \end{align}
    \begin{enumerate}
        \item The $\ue{p}$ are approximate solutions to \eqref{wave_mem_adj}, with the choice
        \begin{align}\label{q0}
        q^0(x):=\frac{1}{M-M^3\varepsilon^2}\ue{p}(0,x).
        \end{align}

        \item The initial energy of $\ue{p}$ satisfies
        \begin{align}\label{in_en}
        E^\varepsilon(0):=E(\ue{p})(0) = 1+\mathcal O(\sqrt{\varepsilon}),
        \end{align}
        i.e. it is bounded as $\varepsilon\to 0$.

        \item The energy of $\ue{p}$ is exponentially small off the vertical ray $(t,x_0)$:
        \begin{align}\label{en_ray}
        \int_{|x-x_0|>\varepsilon^{\frac 18}} \left(|\ue{p}_x|^2 + |\ue{p}_t|^2\right)\,{\rm dx} = \mathcal O\big(e^{-2\varepsilon^{-\frac 14}}\big).
        \end{align}
    \end{enumerate}
\end{theorem}

\begin{proof}
All the properties are obtained through direct computations, employing the definition \eqref{pe} of $\ue{p}$. First of all, we can readily check that
\begin{itemize}
    \item $\ue{p}_{tt}(t,x) = \left(M-M^3\varepsilon^2\right)^2\ue{p}(t,x)$,
	\\
    \item $\displaystyle\ue{p}_{xx}(t,x) = \left[\left(\frac i\varepsilon - \frac{2}{\sqrt{\varepsilon}}(x-x_0)\right)^2 - \frac{2}{\sqrt{\varepsilon}}\right]\ue{p}(t,x)$,
	\\
    \item $\displaystyle\int_0^t \ue{p}_{xx}(s,x)\,{\rm ds} = \frac{1}{M-M^3\varepsilon^2}\left[\left(\frac i\varepsilon - \frac{2}{\sqrt{\varepsilon}}(x-x_0)\right)^2 - \frac{2}{\sqrt{\varepsilon}}\right]\Big(\ue{p}(t,x) - p^{\varepsilon,0}(x)\Big)$,
	\\ 
    \item $\displaystyle q^0_{xx}(x) = \frac{1}{M-M^3\varepsilon^2}\left[\left(\frac i\varepsilon - \frac{2}{\sqrt{\varepsilon}}(x-x_0)\right)^2 - \frac{2}{\sqrt{\varepsilon}}\right]p^{\varepsilon,0}(x)$.
\end{itemize}
In view of that, we have
\begin{align*}
    \ue{p}_{tt} &- \ue{p}_{xx} + M\int_0^t\ue{p}_{xx}\,{\rm ds} + Mq^0_{xx}
    \\
    &= \left[M^2 + \mathcal O(\varepsilon^2) - \frac{M^3\varepsilon^2}{M-M^3\varepsilon^2}\left(\frac i\varepsilon - \frac{2}{\sqrt{\varepsilon}}(x-x_0)\right)^2 + \frac{M^3\varepsilon^2}{M-M^3\varepsilon^2}\frac{2}{\sqrt{\varepsilon}}\right]\ue{p}(t,x)
    \\
    &= \left[M^2 + \mathcal O(\varepsilon^2) + \mathcal O(\varepsilon^{\frac 12}) + \mathcal O(\varepsilon^{\frac 32})  \right]\ue{p}(t,x) = \mathcal O(\varepsilon^{\frac{7}{8}}),
\end{align*}
taking into account the scaling of $\ue{p}$ with respect to $\varepsilon$. Hence, $\ue{p}$ is an approximate solution of \eqref{wave_mem_adj}.

In order to prove the conservation of the energy \eqref{in_en}, let us firstly recall that $E^\varepsilon(t)$ is classically defined through the following integral
\begin{align*}
    E^\varepsilon(t) := \frac 12 \int_\RR \left(|\ue{p}_t(t,x)|^2 + |\ue{p}_x(t,x)|^2\right)\,{\rm dx}.
\end{align*}
We then have
\begin{align*}
    E^\varepsilon(0) &= \frac 12 \int_\RR \left(|\ue{p}_t(0,x)|^2 + |\ue{p}_x(0,x)|^2\right)\,{\rm dx} 
    \\
    &= \frac{\varepsilon^{\frac 74}}{x_0^2\sqrt{2\pi}} \int_\RR\! e^{-\frac{2}{\sqrt{\varepsilon}}(x-x_0)^2}\!\left(\frac{x^2}{\varepsilon^2} + \frac 1\varepsilon(x-x_0)^4 + \left(M-M^3\varepsilon^2\right)^2\right)\,{\rm dx}.
\end{align*}

The three integrals appearing in the above expression have already been computed above. Substituting their values we find
\begin{align*}
    E^\varepsilon(0) &= \frac{1}{x_0^2\sqrt{2\pi}}\varepsilon^{\frac 74} \left(\frac{x_0^2\sqrt{\pi}}{\sqrt{2}}\varepsilon^{-\frac 74} + \frac{\sqrt{\pi}}{4\sqrt{2}}\varepsilon^{-\frac 54} + \frac{3\sqrt{\pi}}{16\sqrt{2}}\varepsilon^{\frac 14} + \left(M-M^3\varepsilon^2\right)^2\frac{\sqrt{\pi}}{\sqrt{2}}\varepsilon^{\frac 14} \right) = \frac 12 + \mathcal O(\varepsilon^{\frac 12}).
\end{align*}
Let us now conclude with the proof of \eqref{en_ray}. We have
\begin{align*}
    &\int_{|x-x_0|>\varepsilon^{\frac 18}} \left(|\ue{p}_x|^2 + |\ue{p}_t|^2\right)\,{\rm dx}
    \\
    &\quad= \frac{\varepsilon^{\frac 74}}{x_0^2\sqrt{2\pi}}e^{2(M-M^3\varepsilon^2)t}\int_{|x-x_0|>\varepsilon^{\frac 18}} e^{-\frac{2}{\sqrt{\varepsilon}}(x-x_0)^2}\left(\frac{x^2}{\varepsilon^2} + \frac 1\varepsilon(x-x_0)^4 + \left(M-M^3\varepsilon^2\right)^2\right)\,{\rm dx}
\end{align*}
Moreover, we can easily compute
\begin{align*}
    \int_{|x-x_0|>\varepsilon^{\frac 18}} & \frac{x^2}{\varepsilon^2}e^{-\frac{2}{\sqrt{\varepsilon}}(x-x_0)^2}\,{\rm dx}
    \\
    &= \frac{\varepsilon^{-\frac 74}}{\sqrt{2}} \int_{|z|>\sqrt{2}\varepsilon^{-\frac 18}} \left(x_0^2+\sqrt{2}\varepsilon^{\frac 14}x_0z+\frac{\varepsilon^{\frac 12}}{2}z^2\right)e^{-z^2}\,{\rm dz}
    \\
	&= \frac{\varepsilon^{-\frac 74}}{\sqrt{2}} \int_{\sqrt{2}\varepsilon^{-\frac 18}}^{+\infty} \left(2x_0^2 + \varepsilon^{\frac 12}z^2\right)e^{-z^2}\,{\rm dz}
    \\
    &= \frac{x_0^2\sqrt{\pi}}{\sqrt{2}}\varepsilon^{-\frac 74}\erfc\left(\sqrt{2}\varepsilon^{-\frac 18}\right) + \frac 12\varepsilon^{-\frac{11}{8}}e^{-2\varepsilon^{-\frac 14}} + \frac{\sqrt{\pi}}{4\sqrt{2}}\varepsilon^{-\frac 54}\erfc\left(\sqrt{2}\varepsilon^{-\frac 18}\right),
\end{align*}
where with $\erfc(\cdot)$ we indicate the complementary error function. In addition
\begin{align*}
    \int_{|x-x_0|>\varepsilon^{\frac 18}}  \frac 1\varepsilon(x-x_0)^4 e^{-\frac{2}{\sqrt{\varepsilon}}(x-x_0)^2} \,{\rm dx} &= \frac{\varepsilon^{\frac 14}}{\sqrt{2}}\int_{|z|>\sqrt{2}\varepsilon^{-\frac 18}} \frac {z^4}{2}e^{-z^2} \,{\rm dz}
    \\
    &= \varepsilon^{-\frac 18}e^{-2\varepsilon^{-\frac 14}} + \frac 34\varepsilon^{\frac 18}e^{-2\varepsilon^{-\frac 14}} + \frac{3\sqrt{\pi}}{8\sqrt{2}}\varepsilon^{\frac 14}\erfc\left(\sqrt{2}\varepsilon^{-\frac 18}\right).
\end{align*}
Finally,
\begin{align*}
    \int_{|x-x_0|>\varepsilon^{\frac 18}} \left(M-M^3\varepsilon^2\right)^2 e^{-\frac{2}{\sqrt{\varepsilon}}(x-x_0)^2}\,{\rm dx}
 &= \left(M-M^3\varepsilon^2\right)^2\frac{\varepsilon^{\frac 14}}{\sqrt{2}}\int_{|z|>\sqrt{2}\varepsilon^{-\frac 18}}  e^{-z^2}\,{\rm dz} 
    \\
    &= \left(M-M^3\varepsilon^2\right)^2\frac{\sqrt{\pi}}{\sqrt{2}}\varepsilon^{\frac 14}\erfc\left(\sqrt{2}\varepsilon^{-\frac 18}\right).
\end{align*}
Adding all the components, we then get
\begin{align*}
    \int_{|x-x_0|>\varepsilon^{\frac 18}} \left(|\ue{p}_x|^2 + |\ue{p}_t|^2\right)\,{\rm dx}  e^{2(M-M^3\varepsilon^2)t}\left[\erfc\left(\sqrt{2}\varepsilon^{-\frac 18}\right)\left(1 + \mathcal O(\varepsilon^{\frac 12}) \right) + \mathcal O\Big(e^{-2\varepsilon^{-\frac 14}}\Big) \right].
\end{align*}

Moreover, it is well-known that the complementary error function has the following asymptotic expansion (see, e.g., \cite[Section 7.1.2]{olver2010nist})
\begin{align*}
    \erfc(z) \sim \frac{e^{-z^2}}{\sqrt{\pi}z}\sum_{m=0}^{+\infty} (-1)^m\frac{1\cdot3\cdot5\cdots(2m-1)}{2^mz^{2m}}, \quad \textrm{ as } z\to +\infty.
\end{align*}
Hence, in our case we have
\begin{align*}
    \erfc\left(\sqrt{2}\varepsilon^{-\frac 18}\right) \sim \frac{\varepsilon^{\frac 18}}{\sqrt{\pi}} e^{-2\varepsilon^{-\frac 14}}\sum_{m=0}^{+\infty} (-1)^m\frac{1\cdot3\cdot5\cdots(2m-1)}{2^m}\varepsilon^{\frac m4} = \mathcal O\Big(e^{-2\varepsilon^{-\frac 14}}\Big),
\end{align*}
and, finally,
\begin{align*}
    \int_{|x-x_0|>\varepsilon^{\frac 18}} \left(|\ue{p}_x|^2 + |\ue{p}_t|^2\right)\,{\rm dx} = \mathcal O\Big(e^{-2\varepsilon^{-\frac 14}}\Big).
\end{align*}
\end{proof}

\section*{Acknowledgments}

This project has received funding from the European Research Council (ERC) under the European Unions Horizon 2020 research and innovation programme (grant agreement No. 694126-DyCon). The work of the first author was partially supported by the Grants MTM2014-52347, MTM2017-92996 and MTM2017-82996-C2-1-R COSNET of MINECO (Spain), by the ELKARTEK project KK-2018/00083 ROAD2DC of the Basque Government, and by the Grant FA9550-18-1-0242 of AFOSR. Part of this work was done during the second author visit to DeustoTech, and he would like to thank the members of the Chair of Computational Mathematics for their kindness and warm hospitality. The authors wish to thank Prof. Enrique Zuazua (Deusto Tech-University of Deusto, Universidad Autnoma de Madrid and Université Pierre et Marie Curie, Paris) for interesting discussions on the topic of this paper.


\begin{thebibliography}{10}

    \bibitem{Bugeaud} {\sc Y.~Bugeaud}, {\em On simultaneous rational approximation to a real number and its integral powers}, Ann. Inst. Fourier, Vol.~60, No.~6 (2010), pp.~2165-2182.     

    \bibitem{chaves2014null}
	{\sc F.~W.~Chaves-Silva, L.~Rosier, and E.~Zuazua}, {\em Null controllability of a system of viscoelasticity with a moving control}, J. Math. Pures Appl., Vol.~101, No.~9 (2014), pp.~198-222.
	
	\bibitem{chaves2017controllability}
	{\sc F.~W. Chaves-Silva, X.~Zhang, and E.~Zuazua}, {\em Controllability of evolution equations with memory}, SIAM J. Control Optim., Vol.~55 (2017), pp.~2437-2459.
	
	\bibitem{colli2004parallel}
	{\sc P.~Colli Franzone and L.~F.~Pavarino}, {\em A parallel solver for reaction--diffusion systems in computational electrocardiology},
	Math. Models Methods Appl. Sci., Vol.~14, No.~6 (2004), pp.~883-911.

	\bibitem{kahane1962pseudo}
	{\sc J.-P. Kahane}, {\em Pseudo-p\'eriodicit\'e et s\'eries de Fourier lacunaires}, Ann. Sci. \'Ecole Norm. Sup., Vol.~79, No.~3 (1962), pp.~93-150.
	
	\bibitem{kim1993control}
	{\sc J.~U. Kim}, {\em Control of a second-order integro-differential equation}, SIAM J. Control Optim., Vol.~31 (1993), pp.~101-110.
	
	\bibitem{leugering1984exact}
	{\sc G.~Leugering}, {\em Exact controllability in viscoelasticity of fading memory type}, Appl. Anal., Vol.~18 (1984), pp.~221-243.
	
	\bibitem{leugering1987exact}
	{\sc G.~Leugering}, {\em Exact boundary controllability of an integrodifferential equation}, Appl. Math. Optim., Vol.~15 (1987), pp.~223-250.
	
	\bibitem{lions1988exact}
	{\sc J.-L.~Lions}, {\em Exact controllability, stabilization and perturbations for distributed systems}, SIAM Rev., Vol.~30 (1988), pp.~1-68.

	\bibitem{loreti2012boundary}
	{\sc P.~Loreti, L.~Pandolfi, and D.~Sforza}, {\em Boundary controllability and observability of a viscoelastic string}, SIAM J. Control Optim., Vol.~50 (2012), pp.~820-844.
	
	\bibitem{loreti2010reachability}
	{\sc P.~Loreti and D.~Sforza}, {\em Reachability problems for a class of integro-differential equations}, J. Differential Equations, Vol.~248 (2010), pp.~1711-1755.
	
	\bibitem{lu2017null}
	{\sc Q.~L{\"u}, X.~Zhang, and E.~Zuazua}, {\em Null controllability for wave equations with memory}, J. Math. Pures Appl., Vol.~108 (2017), pp.~500-531.
	
	\bibitem{martin2013null}
	{\sc P.~Martin, L.~Rosier, and P.~Rouchon}, {\em Null controllability of the structurally damped wave equation with moving control}, SIAM J. Control Optim., Vol.~51 (2013), pp.~660-684.
	
	\bibitem{mustafa2015control}
	{\sc M.~I. Mustafa}, {\em On the control of the wave equation by memory-type boundary condition}, Discr. Cont. Dyn. Syst., Vol.~35 (2015), pp.~1179-1192.
	
	\bibitem{olver2010nist}
	{\sc F.~W. Olver, D.~W. Lozier, R.~F. Boisvert, and C.~W. Clark}, {\em NIST handbook of mathematical functions}, US Department of Commerce, National Institute of Standards and Technology and Cambridge University Press, 2010.
	
	\bibitem{pandolfi2013boundary}
	{\sc L.~Pandolfi}, {\em Boundary controllability and source reconstruction in a viscoelastic string under external traction}, J. Math. Anal. Appl., Vol.~407 (2013), pp.~464-479.
	
	\bibitem{pruss2013evolutionary}
	{\sc J.~Pr{\"u}ss}, {\em Evolutionary integral equations and applications}, vol.~87, Birkh{\"a}user, 2013.
	
	\bibitem{ralston1982gaussian}
	{\sc J.~Ralston}, {\em Gaussian beams and the propagation of singularities}, Studies in partial differential equations, Vol.~23 (1982), pp.~206-248.
	
	\bibitem{rauch2005polynomial}
	{\sc J.~Rauch, X.~Zhang, and E.~Zuazua}, {\em Polynomial decay for a hyperbolic--parabolic coupled system}, J. Math. Pures Appl., Vol.~84 (2005), pp.~407-470.
	
	\bibitem{renardy1987mathematical}
	{\sc M.~Renardy, W.~J. Hrusa, and J.~A. Nohel}, {\em Mathematical problems in viscoelasticity}, Vol.~35 of Pitman Monographs and Surveys in Pure and Applied Mathematics, Longman Scientific \& Technical, Harlow; John Wiley \& Sons, Inc., New York, 1987.
	
	\bibitem{romanov2013exact}
	{\sc I.~Romanov and A.~Shamaev}, {\em Exact controllability of the distributed system, governed by string equation with memory}, J. Dyn. Control Syst., Vol.~19 (2013), pp.~611-623.Boyer, F. (2013, December). On the penalised HUM approach and its applications to the numerical approximation of null-controls for parabolic problems. In ESAIM Proceedings (Vol. 41, pp. 15-58).
	
	\bibitem{rosier2013unique}
	{\sc L.~Rosier and B.-Y. Zhang}, {\em Unique continuation property and control for the Benjamin-Bona-Mahony equation on a periodic domain}, J.	Differential Equations, Vol.~254 (2013), pp.~141-178.
	
	\bibitem{tucsnak2009observation}
	{\sc M.~Tucsnak and G.~Weiss}, {\em Observation and control for operator semigroups}, Springer Science \& Business Media, 2009.
	
	\bibitem{young2001introduction}
	{\sc R.~M. Young}, {\em An Introduction to Non-Harmonic Fourier Series}, Elsevier, 2001.	
\end{thebibliography}

\end{document}